\documentclass[11pt]{amsart}            

\usepackage{amsmath,amsfonts,amsthm,amssymb}%AMS数学符号字体等的宏包
\usepackage[numbers,sort&compress]{natbib}

\usepackage{amsaddr}

\usepackage{hyperref}
\hypersetup{hidelinks}

\usepackage{geometry}                        %页面设置          

\usepackage{fullpage}

\usepackage{graphicx}                          %插图
\usepackage{setspace}  
\usepackage[font=footnotesize,labelfont=bf]{caption}
\usepackage{placeins}
\usepackage{listings}
\usepackage{tabularx}
\usepackage{xcolor}
\usepackage{indentfirst}
\usepackage{dsfont}

\usepackage{multirow}
\usepackage{overpic}
\usepackage{diagbox}

\usepackage{verbatim}

\usepackage{enumerate}

\usepackage{booktabs}
\usepackage{adjustbox}
\usepackage{siunitx}

\usepackage[normalem]{ulem}

\usepackage{tikz}
\usepackage{hyperref}
\usepackage[capitalise,nameinlink]{cleveref}
\usetikzlibrary{positioning}
\usepackage[title]{appendix}

\usepackage{bm}
\usepackage{svg}
\usepackage{booktabs}

\makeatletter
\@addtoreset{equation}{section}
\makeatother

\numberwithin{figure}{section}

\newtheorem{proposition}{Proposition}[section]
\newtheorem{remark}{Remark}[section]

\newtheorem{lemma}[proposition]{Lemma}

\newcommand\dd{\mathrm{d}}
\newcommand\pp{\partial}

\newcommand\x{\bm{x}}
\newcommand\uvec{\mathbf{u}}
\newcommand\X{\mathbf{X}}

\usepackage{mhchem}

\begin{document}

\title[Detailed-Balance Master-Equation Discretizations]{Structure-Preserving Detailed-Balance Master-Equation Discretizations for Fokker--Planck Equations}

\author{Satish Chandran \and Yiwei Wang$^*$}
\address{Department of Mathematics, University of California Riverside, Riverside, CA 92521, United States}
\email[Corresponding author]{yiweiw@ucr.edu}

\begin{abstract}
We develop a variational--Markov construction of detailed-balance
master-equation discretizations for Fokker--Planck equations directly from
their energy--dissipation laws. Rather than discretizing the differential
operator, we represent mass transfer between neighboring grid points by two
directional jump rates. On each edge, the discrete energy--dissipation law
determines the net flux, while detailed balance fixes the rate ratio; a local
reconstruction of edge quantities from neighboring grid-point values then
uniquely determines both rates. A logarithmic-mean reconstruction recovers the classical Scharfetter--Gummel/Wang--Peskin--Elston rates, while alternative
reconstructions yield other reversible schemes. The resulting semi-discrete
systems conserve mass, preserve nonnegativity, satisfy detailed balance, and
dissipate a discrete free energy. The same edgewise construction extends to state-dependent and degenerate mobilities, nonlocal interaction energies and higher-dimensional problems. The corresponding implicit and semi-implicit fully discrete schemes are linear, mass-conservative, positivity-preserving, and energy-stable, with a time-step restriction required only for nonlocal interaction energies whose kernel is not negative semidefinite. Numerical experiments confirm the expected spatial convergence rates of all detailed-balance schemes and verify their equilibrium accuracy,
positivity preservation, and free-energy decay with discontinuous potentials,
saturation, nonlocal interactions, and two-dimensional problems.

\end{abstract}

\maketitle

\section{Introduction}\label{sec:intro}
Fokker--Planck-type equations are widely used to describe the evolution of a conserved quantity under diffusion and drift, with applications in statistical physics, chemical kinetics, semiconductor transport, and biophysics \cite{keller1971model, schienbein1993langevin, Gillespie2002, epshteyn2022stochastic, Du2005, carrillo2019aggregation,Fitzgerald2023, Zakhleniuk2025}.
In this work, we focus on a thermodynamically consistent subclass of Fokker--Planck-type equations that admits a variational structure, given by
\begin{equation}\label{FP_general}
  \rho_t = \nabla \cdot \big( m(\rho, x) \nabla \mu\big), 
  \qquad 
  \mu = \frac{\delta \mathcal{F}}{\delta \rho}, \quad x \in \Omega
\end{equation}
where $\rho \geq 0$, $\mathcal{F}$ is a free-energy functional of $\rho$, $\mu$ is the associated chemical potential, and $m(\rho, x) \geq 0$ is a mobility function. Under a no-flux or periodic boundary condition, the total mass $\int_{\Omega} \rho ~\dd x$ is conserved, and the system satisfies an energy-dissipation law (EDL) 
  \begin{equation}\label{ED_mu_1}
  \frac{\dd}{\dd t} \mathcal{F}[\rho] = - \int_{\Omega} m(\rho, x) |\nabla \mu|^2 \dd x \leq 0. 
\end{equation}
A prototypical example is the linear Fokker--Planck equation 
\begin{equation}\label{FP_linear}
 \rho_t = \nabla \cdot \left( \rho \nabla (\ln \rho + V(x)) \right),
\end{equation}
in which the free energy and the mobility function are given by 
\begin{equation}\label{eq:intro-energy}
  \mathcal{F}[\rho]=\int_{\Omega} \bigl( \rho (\ln \rho - 1) + \rho\, V(x) \bigr) \,\dd x,
  \qquad m(\rho, x) = \rho.
\end{equation}
Here, $V(x)$ is an external potential. Other choices of the free energy and of the mobility lead to nonlinear and nonlocal equations, such as the porous-medium equation \cite{otto2001geometry}, the Keller--Segel equation \cite{keller1971model}, and the Poisson--Nernst--Planck equation \cite{eisenberg2010energy}.

It is often challenging to construct structure-preserving discretizations of Fokker--Planck-type equations \eqref{FP_general}, since such a discretization must preserve mass conservation and the nonnegativity of $\rho$, dissipate the discrete free energy, and reproduce the correct equilibrium \cite{Bubba2019}. These requirements are related but not equivalent.
A method may dissipate a discrete free energy and be accurate over short times while converging to a wrong or inaccurate equilibrium; conversely, accurately reproducing the equilibrium does not by itself imply comparable accuracy in the transient dynamics. % This distinction is particularly consequential for drift-dominated or discontinuous potentials\cite{wang2003robust,wang2007convergence,li2020large}.

Over the past decades, many structure-preserving schemes for Fokker--Planck-type equations have been developed \cite{SG1969,CHANG1970,wang2003robust, shen2020unconditionally, li2020large, liu2020lagrangian, heida2021consistency, cances2024discretizing}, which address the numerical challenges from different perspectives. One of the most widely used classes can be written
in exponential-fitting form and includes the Il'in--Allen--Southwell
schemes, originating from steady viscous-flow and singular-perturbation
problems \cite{AllenSouthwell1955,Iln1969}; the Scharfetter--Gummel (SG)
scheme for semiconductor drift--diffusion equations \cite{SG1969}; and the
Chang--Cooper (CC) \cite{CHANG1970} and Wang--Peskin--Elston (WPE)
schemes \cite{wang2003robust} for Fokker--Planck equations. Although these
methods were originally derived from different viewpoints, in the
one-dimensional linear constant-diffusion setting, with the drift
approximated as constant on each grid interval, they reduce to the same Bernoulli-function numerical flux.
In \cite{wang2003robust}, the WPE scheme was interpreted as the master equation of a discrete-state Markov chain whose forward and backward jump rates satisfy the detailed-balance condition with respect to the discrete Gibbs state \cite{wang2003robust}. Its convergence for discontinuous potentials and a corresponding improved variant were subsequently studied in \cite{wang2007convergence,wang2008convergence}. Consistency and error estimates for a family of finite-volume discretizations that includes the SG scheme were established in \cite{heida2021consistency}, while the large-time behavior of a related class of $B$-schemes was studied from a jump-process perspective in \cite{li2020large}. SG- and CC-type constructions have also been extended to nonlinear and nonlocal Fokker--Planck, aggregation--diffusion, and Keller--Segel-type equations \cite{BessemoulinChatard2012NM,BessemoulinChatard2012, Pareschi2018,Bubba2019, schlichting2022scharfetter, hraivoronska2023,  jungel2026generalized}. The same SG-type schemes can also be derived using the Slotboom, or non-logarithmic Landau, transform \cite{liu2018positivity, heida2021consistency,liu2020positive, liu2021efficient}.
Another closely related family consists of nonlinear upwind schemes
\cite{carrillo2015finite,chow2019entropy,BailoCarrilloHu2020}. Their
gradient-flow formulation was analyzed alongside SG-type discretization within a
common energy--dissipation framework in \cite{Bubba2019}. These schemes
first discretize the chemical potential and then upwind the density or
mobility according to the sign of its discrete difference. The numerical
flux is chosen so that each edge contributes nonpositively to the change
of the discrete energy, yielding a discrete energy--dissipation law.
When the edge mobilities are nondegenerate, zero dissipation is equivalent
to constancy of the discrete chemical potential on each connected
component and to vanishing edge fluxes. Equilibrium is thus characterized
without introducing two directional transition rates or an explicit
detailed-balance relation between them. Convergence, second-order
reconstructions, and bound-preserving extensions have been studied in
\cite{BailoCarrilloMurakawaSchmidtchen2020,BailoCarrilloHu2020,
BailoCarrilloHu2023}.

Despite these developments, most existing constructions start from a
particular representation of the partial differential equation, such as a
local stationary problem, a change of variable form, or a prescribed numerical
flux. Their derivations are therefore typically tied to the specific form of the equation and must be adapted when the mobility or free energy changes. A
related conceptual issue is that the distinct roles of the
discrete energy--dissipation law and detailed balance are not always made explicit.
The former constrains the transient dynamics through discrete dissipation,
whereas the latter encodes equilibrium information.

To address these issues, we propose a variational--Markov construction that
does not directly discretize the differential operator in
\eqref{FP_general}. Instead, the probability density is represented by a
finite-dimensional probability vector, and its mass-conserving evolution is
written as a master equation with unknown directional jump rates. The free energy and dissipation are discretized using the same
finite-dimensional probability representation. The two
structural conditions enter separately on each neighboring pair
(Fig.~\ref{fig:overview}): the discrete energy--dissipation law determines
the net flux, while detailed balance determines the ratio of the two
directional rates. Once this ratio is fixed, the two rates share a common edge prefactor. Since only grid-point values are available, the edge quantities entering this prefactor are reconstructed
locally from adjacent grid values. Together,
these relations uniquely determine the two rates. A logarithmic-mean
reconstruction recovers the classical SG/WPE rates, while alternative local
means yield other reversible schemes within the same formulation. Consistency is established through the consistency
of the discrete variational quantities and the reconstructed edge values.
Because the derivation depends only on these quantities and the neighboring
relations on the grid, the same argument extends naturally to more complex
Fokker--Planck-type systems, including state-dependent and degenerate
mobilities, nonlocal interaction energies, general \(f\)-divergence free
energies, and multiple dimensions. The resulting
decomposition also separates equilibrium from dynamics: the detailed-balance
ratio determines the discrete equilibrium, whereas the
common edge prefactor controls the transient evolution.

\begin{figure}[!ht]
\centering
\resizebox{\linewidth}{!}{%
\begin{tikzpicture}[
  font=\footnotesize,
  box/.style={
    rectangle,
    rounded corners=2pt,
    draw=blue!45!black,
    fill=blue!3,
    line width=0.7pt,
    align=center,
    inner xsep=6pt,
    inner ysep=5pt
  },
  flow/.style={
    ->,
    line width=0.7pt,
    draw=black!70
  }
]

\node[box, text width=60mm] (projection) {
  \textbf{Variational projection}\\[2pt]
  %$(\mathcal M, \mathcal F,\triangle) \longmapsto (\mathcal M_h, \mathcal F_h,\triangle_h)$\\
  Fokker--Planck
  $\longmapsto$
  master equation
};

\node[box, text width=29mm, right=7mm of projection] (edl) {
  \textbf{Discrete EDL}
};

\node[box, text width=29mm, above=4mm of edl] (db) {
  \textbf{Detailed balance}
};

\node[box, text width=29mm, below=4mm of edl] (recon) {
  \textbf{Reconstruction}
};

\node[
  box,
  fill=blue!7,
  text width=50mm,
  right=9mm of edl
] (master) {
  \textbf{Reversible master equation}\\[2pt]
  % directional jump rates\\[2pt]
  linear implicit or semi-implicit temporal discretization
};

\draw[flow] (projection.east) -- (edl.west);
\draw[flow] (projection.east) -- (db.west);
\draw[flow] (projection.east) -- (recon.west);

\draw[flow] (edl.east) -- (master.west);
\draw[flow] (db.east) -- (master.west);
\draw[flow] (recon.east) -- (master.west);

\end{tikzpicture}%
}
\caption{
Overview of the variational--Markov construction. The discrete energy--dissipation law determines the net flux, while detailed
balance determines the rate ratio. The symmetric edge coefficient is evaluated through
a local reconstruction from the adjacent grid-point values.  Together, these ingredients specify the two directional jump
rates in the discrete master equation, whose implicit or semi-implicit discretizations lead to linear, positivity-preserving, and
free-energy dissipative schemes for the Fokker--Planck-type equations.
}
\label{fig:overview}
\end{figure}
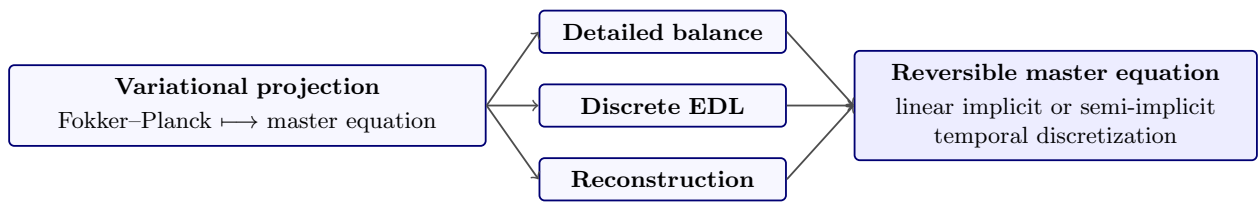

The construction also leads to simple fully discrete schemes. For the linear fixed-potential problem, backward Euler
produces a nonsingular \(M\)-matrix system. The update is therefore uniquely solvable, preserves mass and nonnegativity, and dissipates the discrete free
energy for every time-step size. For state-dependent mobilities, freezing the edge coefficients at the previous time level retains a linear,
unconditionally structure-preserving update, even when the mobility is degenerate. For nonlocal interaction energies, a semi-implicit treatment still yields a uniquely solvable linear system that preserves mass and nonnegativity without a time-step restriction; the discrete energy
inequality holds under an explicit restriction arising from the nonlocal interaction term. No nonlinear iteration is required in any of these cases.

The rest of the paper is organized as follows.
Section~\ref{sec:structure} reviews the variational structures of
Fokker--Planck equations and detailed-balance master equations, as well as their connections.
Section~\ref{sec:rates} derives the jump rates and identifies the schemes
associated with different reconstructions, and
Section~\ref{sec:temporal} discusses their temporal discretizations.
Section~\ref{sec:extensions} develops extensions to nonlinear mobilities,
nonlocal interactions, and higher-dimensional problems.
Section~\ref{sec:numerics} examines transient accuracy, equilibrium accuracy,
and structure preservation through five numerical experiments.

\section{Preliminaries}		\label{sec:structure}
In this section, we briefly review the variational structure of Fokker--Planck equations and master equations for discrete Markov chains, as well as their connections.

\subsection{Variational Structure of Fokker--Planck-type equations}

The energy-dissipation law introduced in the Introduction is not just a property of Fokker--Planck-type equations; together with the admissible set of $\rho$, denoted by $\mathcal{M}$, it in fact determines the evolution equation itself. 
More generally, a thermodynamically consistent model for a mechanically closed,
isothermal system can be specified by an energy-dissipation triple $(\mathcal M,\mathcal F,\triangle),$ together with a choice of variational variables and the associated
kinematics, under the framework of the energetic
variational approach \cite{giga2018variational, wang2022some}, a variational principle motivated by non-equilibrium thermodynamics \cite{onsager1931reciprocal, onsager1931reciprocal2}. Here, $\mathcal M$ specifies the admissible set of states. The kinematics relates the variational variables to the state variables and ensures that the resulting evolution remains in the admissible set $\mathcal M$.
$\mathcal F$ is the free-energy functional, and $\triangle\geq0$ is the rate of energy dissipation that is expressed in terms of the state and the rates of
the underlying variational variables \cite{giga2018variational}. The relation between $\mathcal{F}$ and $\triangle$ can be expressed as an energy-dissipation law
\begin{equation}\label{eq:abstract_edl}
    \frac{\dd}{\dd t}\mathcal F=-\triangle.
\end{equation}

The Fokker--Planck equation \eqref{FP_general} belongs to a broader
class of models, known as generalized diffusions \cite{giga2018variational}.  These are models that describe the evolution of a conserved quantity $\rho$. The admissible set of states can be written
formally as $\mathcal M = \{
        \rho\geq0:
        \int_\Omega\rho\,\dd\x= M
    \},$
where $M$ is the total mass. The variational variable is the
flow map $\x=\x(\X,t)$ that transports the initial density $\rho_0(\X)$ to the current density $\rho(\x,t)$, which is defined by
\begin{equation}\label{eq:flow_map}
    \frac{\dd}{\dd t}\x(\X,t)
    =
    \uvec(\x(\X,t),t),
    \qquad
    \x(\X,0)=\X,
\end{equation}
where $\uvec$ is a velocity field, and $\X$ and $\x$ denote the Lagrangian and Eulerian coordinates,
respectively. Conservation of mass gives the kinematic relation
\begin{equation}\label{eq:lagrangian_mass}
    \rho(\x(\X,t),t) =
    \rho_0(\X) / \det{\sf F}(\X,t),
    \qquad
    {\sf F}(\X,t)
    =
    \nabla_{\X}\x(\X,t)
    .
\end{equation}
Provided that the flow map is orientation preserving,
\eqref{eq:lagrangian_mass} preserves both the nonnegativity of $\rho$
and its total mass. So the constraints in $\mathcal{M}$ are encoded in \eqref{eq:lagrangian_mass}.
In Eulerian coordinates, (\ref{eq:lagrangian_mass}) can be formally
expressed as the continuity equation
\begin{equation}\label{eq:continuity_flux}
    \rho_t+\nabla\cdot(\rho\uvec)=0.
\end{equation}

After specifying the kinematics in \eqref{eq:continuity_flux}, the choice of the free energy $\mathcal F$ and of the dissipation rate $\triangle$ \emph{determines} the equation for the velocity $\uvec$, and hence the evolution of $\rho$.
A typical choice of the free energy and the rate of energy dissipation is
\begin{equation}\label{eq:nonlocal_energy}
\begin{aligned}
    \mathcal F[\rho]
    &=
    \int_\Omega
        \bigl(
            \omega(\rho)+\rho V(x)
        \bigr)
    \,\dd x      +
    \frac12
    \int_\Omega\int_\Omega
        K(x,x')\rho(x)\rho(x')
    \,\dd x\,\dd x',
\end{aligned}
\end{equation}
and
\begin{equation}\label{eq:physical_dissipation}
    \triangle(\rho,\uvec)
    =
    \int_\Omega
        \eta(\rho,x)|\uvec|^2
    \,\dd x,
\end{equation}
where $\omega(\rho)$ is a local free-energy density,
$K(x,x')=K(x',x)$ is a symmetric interaction kernel, and
$\eta(\rho,x)>0$ is a friction coefficient. The quadratic dependence
of $\triangle$ on the rate $\uvec$ corresponds to Darcy's law in physics. Starting from (\ref{eq:nonlocal_energy}) and (\ref{eq:physical_dissipation}), the energetic variational approach combines the least action principle
and the maximum dissipation principle
\cite{giga2018variational} to obtain the force balance equation, i.e., the equation for $\uvec$. In the overdamped setting,
the action is formally given by $\mathcal A[\x] = -\int_0^T\mathcal F[\rho]\,\dd t.$
Varying the action with respect to the flow map, subject to the
kinematic relation \eqref{eq:lagrangian_mass}, gives the conservative
force density
\begin{equation}\label{eq:conservative_force}
    \frac{\delta\mathcal A}{\delta\x}
    =
    -\rho\nabla\mu, \quad  \mu = \frac{\delta \mathcal F}{\delta \rho}.
\end{equation}
On the dissipative side, the physical dissipation rate
\eqref{eq:physical_dissipation} defines the quadratic dissipation
potential
\begin{equation}\label{eq:process_dissipation_potential}
    \Psi_\rho(\uvec)
    =
    \frac12\triangle(\rho,\uvec)
    =
    \frac12
    \int_\Omega
        \eta(\rho,x)|\uvec|^2
    \,\dd x.
\end{equation}
The maximum dissipation principle gives the dissipative force
\begin{equation}\label{eq:dissipative_force}
    \frac{\delta\Psi_\rho}{\delta\uvec}
    =
    \eta(\rho,x)\uvec.
\end{equation}
Finally, the force balance condition yields the generalized
Darcy's law
\begin{equation}\label{eq:force_balance}
    \eta(\rho,x)\uvec
    =
    -\rho\nabla\mu,
    \qquad
    \mu
    =
    \frac{\delta\mathcal F}{\delta\rho}.
\end{equation}
Combining \eqref{eq:force_balance} with the continuity equation gives a FP-type equation
\begin{equation}\label{eq:generalized_diffusion_from_envara}
    \rho_t
    =
    \nabla\cdot
    \left(
        m(\rho,x)\nabla\mu
    \right),
    \qquad
    m(\rho,x)
    =
    \frac{\rho^2}{\eta(\rho,x)}.
\end{equation}
Different choices of
$\mathcal F$ and $\triangle$ generate the broader class of generalized
diffusion equations represented by \eqref{FP_general}.

\begin{remark}
The dual dissipation potential, defined with respect to the force
density $\boldsymbol f$, is
\begin{equation}\label{eq:dual_process_dissipation}
\begin{aligned}
    \Psi_\rho^*(\boldsymbol f)
    &:=
    \sup_{\uvec}
    \left\{
        \int_\Omega
            \boldsymbol f\cdot\uvec
        \,\dd x
        -
        \Psi_\rho(\uvec)
    \right\} =
    \frac12
    \int_\Omega
        \frac{|\boldsymbol f|^2}{\eta(\rho,x)}
    \,\dd x.
\end{aligned}
\end{equation}
Along the force balance
$\boldsymbol f=-\rho\nabla\mu=\eta(\rho,x)\uvec$, the same physical
dissipation can therefore be evaluated in either the rate or force
variables:
\begin{equation}\label{eq:dissipation_primal_dual}
\begin{aligned}
    \triangle
    &=
    2\Psi_\rho(\uvec)
    =
    2\Psi_\rho^*(-\rho\nabla\mu) =
    \int_\Omega
        \eta(\rho,x)|\uvec|^2
    \,\dd x
    =
    \int_\Omega  m(\rho,x)
        |\nabla\mu|^2
    \,\dd x.
\end{aligned}
\end{equation}
\end{remark}

\begin{remark}
In the abstract formulation of generalized gradient systems \cite{mielke2023introduction}, a model is
specified by a triple $\bigl(\mathcal M,\mathcal F,\mathcal R\bigr),$
where the dissipation potential
$\mathcal R(\rho,\dot\rho)$ is defined directly on the tangent bundle of
the state space. The admissible kinematics is therefore absorbed into
the state-space dissipation potential through its dependence on the
state rate $\dot\rho$. The energetic variational approach uses a related but more explicit
process-space formulation. The dissipation potential is naturally
expressed in terms of the rate of an underlying variational variable.
\end{remark}

\begin{remark}
For the linear Fokker--Planck equation \eqref{FP_linear} with unit mass,
the equilibrium density is
\begin{equation}\label{eq:Gibbs}
    \rho^{\mathrm{eq}}(x)
    =
    Z^{-1}e^{-V(x)},
    \qquad
    Z
    =
    \int_\Omega e^{-V(x)}\,\dd x.
\end{equation}
In this case, the free energy \eqref{eq:intro-energy} is equivalent to the Kullback--Leibler (KL) divergence (or relative entropy) between $\rho$ and $\rho^{\mathrm{eq}}$,
\[
    {\rm KL}(\rho\|\rho^{\mathrm{eq}})
    :=
    \int_\Omega
        \rho
        \ln\left(
            \frac{\rho}{\rho^{\mathrm{eq}}}
        \right)
    \,\dd x,
\]
up to a constant. Similarly, the dissipation $\int_\Omega \rho |\nabla \mu|^2 \dd x$ is the relative Fisher information
\[
    \triangle
    =
    I(\rho\|\rho^{\mathrm{eq}})
    :=
    \int_\Omega
        \rho
        \left|
            \nabla
            \ln\left(
                \frac{\rho}{\rho^{\mathrm{eq}}}
            \right)
        \right|^2
    \,\dd x.
\]
\end{remark}

\subsection{Variational structure of Master Equations with detailed balance}\label{subsec:master_var}
While a FP-type equation describes evolution of a probability distribution on a continuous state space, a \emph{master equation} governs the evolution of a probability distribution ${\bm p} = (p_1, \ldots, p_N)$ on a finite state space $\mathcal{V} = \{1, \ldots, N\}$, with $\sum_i p_i = 1$, $p_i \geq 0$. Let $q_{ij}(t)$ denote the transition rate from state $i$ to state $j$. We assume $q_{ij}$ and $q_{ji}$ are either both positive or both zero, and form the finite graph $\mathcal{G} = (\mathcal{V}, \mathcal{E})$ with edge set $\mathcal{E} = \{(i,j) \mid q_{ij} > 0,\, q_{ji} > 0\}$; for $i \in \mathcal{V}$ we write $\mathcal{N}(i) := \{j \in \mathcal{V} \mid (i,j) \in \mathcal{E}\}$ for the neighbors of $i$.
We further assume irreducibility, i.e.\ $\mathcal{G}$ is connected.
The master equation reads
\begin{equation}\label{Master_eq}
\frac{\dd p_i(t)}{\dd t} = \sum_{j \in \mathcal{N}(i)} \left( q_{ji}(t) p_j(t) - q_{ij}(t) p_i(t) \right), \quad \forall i=1,2,\ldots,N.
\end{equation}
When the rates are constant, explicitly time-dependent, or dependent on $\bm{p}(t)$, the process is known as a time-homogeneous, time-inhomogeneous, or nonlinear Markov chain \cite{kolokoltsov2010nonlinear}, respectively. As with the Fokker--Planck equation, the master equation has a variational structure if the transition rates satisfy the {\bf detailed-balance condition}. Let ${\bm\pi} = (\pi_1, \ldots, \pi_N) > 0$ be a positive stationary distribution. We say that the transition rates satisfy the detailed-balance condition if
\begin{equation}\label{DB_1}
  q_{ij} \pi_i = q_{ji} \pi_j, \quad (i, j) \in \mathcal{E}.
\end{equation}

We define the free energy as the discrete $f$-divergence \cite{NEURIPS2020_NengWan} 
\begin{equation}\label{f_div}
\mathcal{F}[{\bm p}] = \sum_{i=1}^{N} f \left( \frac{p_i}{\pi_i} \right) \pi_i,
\end{equation}
where $f: [0,\infty) \to (-\infty, +\infty]$ is convex with $f(1) = 0$; this is enough for $\mathcal{F}$ itself to be well defined. For $f(x) = x\ln x$, this reduces to the KL divergence
\begin{equation}\label{KL_div}
\mathcal{F}[{\bm p}] = \sum_{i=1}^N \left( p_i \ln p_i - p_i \ln \pi_i \right).
\end{equation}
Since $\pi_i > 0$, there exists a potential $V_i$ such that $\pi_i = \frac{1}{Z} e^{-V_i}$, $Z = \sum_{i=1}^N e^{-V_i}$, where $V_i$ is the potential energy of state $i$. The KL divergence is therefore equivalent to the free energy
\begin{equation}\label{free_energy_master}
\mathcal{F}[{\bm p}] = \sum_{i=1}^N \left( p_i \ln p_i + p_i V_i \right)
\end{equation}
up to an additive constant. Under the detailed balance condition (\ref{DB_1}), one can show that the master equation satisfies an energy-dissipation law with respect to the free energy (\ref{free_energy_master}) \cite{schnakenberg1976network,ge2010physical, maas2011gradient, chow2012fokker, chow2019entropy, disser2015gradient}. More precisely, we have the following proposition.

\begin{proposition}\label{prop:master_EDL}
Let $\mathcal{G} = (\mathcal{V}, \mathcal{E})$ be a connected, undirected finite graph with vertex set $\mathcal{V} = \{1, \dots, N\}$. Let $\bm{p}(t) = (p_1(t), \dots, p_N(t))$ be a time-dependent probability distribution satisfying $p_i(t) > 0$ and $\sum_{i=1}^N p_i(t) = 1$ for all $t \geq 0$. Suppose that $\bm{p}(t)$ evolves according to the master equation (\ref{Master_eq}), and the transition rates satisfy the detailed balance condition (\ref{DB_1}). Then, for the $f$-divergence $\mathcal{F}[{\bm p}]$ defined in \eqref{f_div} with convex $f \in C^1(0,\infty)$, the system satisfies the energy-dissipation law
\begin{equation}\label{ED_master_f}
\frac{\dd}{\dd t} \mathcal{F}[{\bm p}]
= - \sum_{(i,j) \in \mathcal{E},\, i < j}
  \ell\!\left( \frac{p_i}{\pi_i}, \frac{p_j}{\pi_j} \right)
  \left( q_{ij}\, p_i - q_{ji}\, p_j \right) \ln \left( \frac{p_i q_{ij}}{p_j q_{ji}} \right) \leq 0,
\end{equation}
where 
\begin{equation}\label{eq:ell_def}
  \ell(x,y):=\frac{f'(x)-f'(y)}{\ln x-\ln y}\ \ (x\neq y)
  % \qquad \ell(x,x):= x f''(x),
\end{equation}
is nonnegative by the convexity of $f$. For the KL divergence $f(x)=x\ln x$ one has $\ell\equiv1$ and \eqref{ED_master_f} reduces to
\begin{equation}\label{ED_master}
\frac{\dd}{\dd t} \mathcal{F}[{\bm p}]
= - \sum_{(i,j) \in \mathcal{E},\, i < j} \left( q_{ij}\, p_i - q_{ji}\, p_j \right) \ln \left( \frac{p_i q_{ij}}{p_j q_{ji}} \right) \leq 0.
\end{equation}
\end{proposition}

\begin{proof}
Since $f \in C^1(0,\infty)$, convexity of $f$ implies $f'$ is non-decreasing. Differentiating $\mathcal{F}[\bm p] = \sum_{i=1}^N f(p_i/\pi_i)\pi_i$ in $t$ and substituting the master equation for $\dot p_i$ gives
\begin{equation}\label{ED_MC1}
  \begin{aligned}
    \frac{\dd}{\dd t} \mathcal{F}[{\bm p}] & = \sum_{i=1}^N f'\left( \frac{p_i}{\pi_i} \right) \dot p_i = \sum_{i=1}^N f'\left( \frac{p_i}{\pi_i} \right) \sum_{j \in \mathcal{N}(i)} \left( q_{ji} p_j  - q_{ij} p_i\right)  = \sum_{(i,j) \in \mathcal{E}} f'\left( \frac{p_i}{\pi_i} \right)   ( q_{ji} p_j - q_{ij} p_i ) \\
    & = \frac{1}{2} \sum_{(i,j) \in \mathcal{E}}  \left(  f'\left( \frac{p_i}{\pi_i} \right)  - f'\left( \frac{p_j}{\pi_j} \right)   \right)   ( q_{ji} p_j - q_{ij} p_i ),
   \end{aligned}
\end{equation}
where the second equality substitutes the master equation \eqref{Master_eq}, and the third reindexes the double sum via $\sum_{i=1}^N\sum_{j\in\mathcal{N}(i)}(\cdot) = \sum_{(i,j)\in\mathcal{E}}(\cdot)$. For the last equality, let $T$ denote the sum in the third expression; since $\mathcal{E}$ is symmetric under $(i,j)\mapsto(j,i)$, relabeling the summation index gives
\[
T = \sum_{(i,j)\in\mathcal{E}} f'\!\left(\frac{p_j}{\pi_j}\right)(q_{ij}p_i - q_{ji}p_j) = -\sum_{(i,j)\in\mathcal{E}} f'\!\left(\frac{p_j}{\pi_j}\right)(q_{ji}p_j - q_{ij}p_i);
\]
adding this identity to the original expression for $T$ and dividing by $2$ gives the stated form. By the detailed balance condition $q_{ij}\pi_i = q_{ji}\pi_j$ for $(i,j)\in\mathcal{E}$,
\[
\frac{p_i q_{ij}}{p_j q_{ji}} = \frac{p_i}{p_j}\cdot\frac{\pi_j}{\pi_i} = \frac{p_i/\pi_i}{p_j/\pi_j},
\qquad\text{so}\qquad
\ln\left(\frac{p_i q_{ij}}{p_j q_{ji}}\right) = \ln\left(\frac{p_i}{\pi_i}\right) - \ln\left(\frac{p_j}{\pi_j}\right).
\]
Recall the definition \eqref{eq:ell_def} of $\ell$. Since $f'$ is non-decreasing and $\ln$ is increasing, $f'(x)-f'(y)$ and $\ln x - \ln y$ have the same sign whenever $x \neq y$, so $\ell \geq 0$ everywhere. With $x = p_i/\pi_i$, $y = p_j/\pi_j$, the two displays above combine to give
\begin{equation}
f'\left( \frac{p_i}{\pi_i} \right)  - f'\left( \frac{p_j}{\pi_j} \right) = \ell\left(\frac{p_i}{\pi_i}, \frac{p_j}{\pi_j}\right) \ln \left( \frac{p_i q_{ij}}{p_j  q_{ji}}  \right),
\end{equation}
which also holds when $p_i/\pi_i = p_j/\pi_j$: by the log identity above, this case forces $q_{ij}p_i = q_{ji}p_j$, so both sides vanish regardless of the value assigned to $\ell(x,x)$. Substituting into \eqref{ED_MC1} yields
\begin{equation}\label{MC_entropy_production}
\begin{aligned}
\frac{\dd}{\dd t} \mathcal{F}[{\bm p}] &=  \frac{1}{2} \sum_{(i,j) \in \mathcal{E}}  \ell\left(\frac{p_i}{\pi_i}, \frac{p_j}{\pi_j}  \right) \ln \left( \frac{p_i q_{ij}}{p_j  q_{ji}}  \right) ( q_{ji} p_j - q_{ij} p_i ) \\
&=  - \frac{1}{2} \sum_{(i, j) \in \mathcal{E}}  \ell\left(\frac{p_i}{\pi_i}, \frac{p_j}{\pi_j}  \right) ( q_{ij} p_i - q_{ji} p_j )  \ln \left( \frac{q_{ij}p_i}{q_{ji} p_j} \right) \leq 0,
\end{aligned}
\end{equation}
where the inequality holds because $(a-b)\ln(a/b) \geq 0$ for all $a,b>0$ (as $a-b$ and $\ln a - \ln b$ share sign); taking $a = q_{ij}p_i$, $b = q_{ji}p_j$ shows each summand $-\ell\,(q_{ij}p_i - q_{ji}p_j)\ln(q_{ij}p_i/(q_{ji}p_j))$ is $\leq 0$, since $\ell \geq 0$. For the KL-divergence case, $f(x) = x\ln x$, we have $f'(x) = 1 + \ln x$, so $\ell \equiv 1$.
\end{proof}

\begin{remark}
  In addition to the KL divergence, another popular choice of $f$ is
  \[
    f(r)=\frac12(r-1)^2,
\]
which leads to the weighted quadratic entropy
\begin{equation}
    \mathcal{Q}[{\bm p}\mid{\bm \pi}]
    :=
    \frac12 \sum_{i=1}^N
    \frac{(p_i-\pi_i)^2}{\pi_i},
\end{equation}
i.e.\ the $\chi^2$-divergence of ${\bm p}$ relative to ${\bm \pi}$. With \( {\bm u}={\bm p}-{\bm \pi}\), the energy-dissipation law for \(\mathcal{Q}\) reads
\begin{equation}
    \frac{\dd}{\dd t}
    \frac12\sum_{i=1}^N \frac{u_i^2}{\pi_i}
    =
    -\frac12
    \sum_{(i,j)\in\mathcal E}
    q_{ij}\pi_i
    \left(
        \frac{u_i}{\pi_i}
        -
        \frac{u_j}{\pi_j}
    \right)^2
    \le 0,
\end{equation}
which can be viewed as a weighted $L^2$-energy for the master equation \cite{li2020large}.
%Thus the weighted quadratic entropy captures the local Hilbert-space dissipative structure of the logarithmic entropy.  This is the mechanism behind many weighted \(L^2\)-type stability estimates for detailed-balanceor \(B\)-scheme discretizations.  % It should, however, be distinguished from the full nonlinear entropy dissipation law: the latter additionally relies on positivity of \({\bm p}\) and the convexity of the logarithmic entropy.
\end{remark}

\subsection{Spatial discretization of a Fokker--Planck equation as a master equation}\label{subsec:trad_disc}

The Fokker--Planck equation and the master equation are closely connected~\cite{xing2005continuum}. On one hand, the FP equation can be viewed as the continuum limit of a master equation when the number of discrete states grows to infinity. On the other hand, a finite difference or finite volume discretization of the FP equation yields a system of ODEs that can be viewed as a master equation~\cite{wang2003robust}. 

To illustrate the idea, we consider the linear Fokker--Planck equation \eqref{FP_linear} with periodic boundary conditions and a periodic potential $V(x)$. We partition the domain $\Omega = (0, 1)$ into $N$ cells of size $h=1/N$ with grid points $x_i = ih$, $i=0,\ldots,N$, and let $\mathcal{I}^h$ denote the space of grid functions on $\{x_i\}_{i=0}^{N}$. We approximate $\rho(x, t)$ by a grid function on $\mathcal{I}^h$, i.e., $\rho_i(t)\approx\rho(x_i,t)$, and define the probability masses
\begin{equation}
p_i(t) = h\,\rho_i(t), \qquad i=0,1,\dots,N-1,
\end{equation}
which satisfy $\sum_i p_i \approx \int_0^1\rho\,\dd x = 1$. We can view $x_i$ ($i = 0, \ldots, N-1$) as $N$ discrete states, and $p_i(t)$ 
represents the probability of finding the particle at state $i$.

A finite difference discretization of \eqref{FP_linear} yields a system of ODEs for $\rho_i(t)$,
\begin{equation}\label{master_rho_sec3}
  \frac{\dd \rho_i}{\dd t} =  F_{i-1/2} \rho_{i-1} - B_{i-1/2} \rho_i - F_{i+1/2} \rho_i + B_{i+1/2} \rho_{i+1}, \quad i = 0, 1, \ldots, N-1,
\end{equation}
with $\rho_N = \rho_0$ and $\rho_{-1} = \rho_{N-1}$.  For example, the standard central difference scheme gives
\begin{equation}
 \frac{\dd}{\dd t}\rho_i = \frac{1}{h}\big(J_{i-\frac{1}{2}} - J_{i+\frac{1}{2}}\big),
 \qquad
 J_{i+\frac{1}{2}} = -\frac{\rho_i+\rho_{i+1}}{2}\frac{V_{i+1}-V_i}{h}
                     - \frac{\rho_{i+1}-\rho_i}{h},
\end{equation}
leading to
\begin{equation} \label{eqn: central difference rates}
F_{i+\frac{1}{2}}^{\mathrm{CD}} = \frac{1}{h^2}\Big(1-\frac{1}{2}\delta V_{i+\frac{1}{2}}\Big), \quad
 B_{i+\frac{1}{2}}^{\mathrm{CD}} = \frac{1}{h^2}\Big(1+\frac{1}{2}\delta V_{i+\frac{1}{2}}\Big),
\end{equation}
with $\delta V_{i+\frac{1}{2}} = V_{i+1}-V_i$. Multiplying (\ref{master_rho_sec3}) by $h$ gives 
\begin{equation}\label{master_p_sec3}
  \frac{\dd p_i}{\dd t} = F_{i-1/2} p_{i-1} - B_{i-1/2} p_i - F_{i+1/2} p_i + B_{i+1/2} p_{i+1}, \quad i = 0, 1, \ldots, N-1,
\end{equation}
with $p_N = p_0$ and $p_{-1} = p_{N-1}$. Equation \eqref{master_p_sec3} is precisely the master equation \eqref{Master_eq} on the cycle graph $\mathcal G_h=(\mathcal V_h,\mathcal E_h)$ with $\mathcal V_h=\{0,1,\ldots,N-1\}$ and nearest-neighbor edges $\mathcal E_h=\{(i,i+1\bmod N)\}$, with transition rates $q_{i,i+1}=F_{i+\frac12}$ and $q_{i+1,i}=B_{i+\frac12}$.

A desirable discretization of the Fokker--Planck equation should recover the correct stationary distribution. The stationary solution of \eqref{FP_linear} is $\rho^{\mathrm{eq}} \propto \exp(-V)$. Correspondingly, the semi-discrete master equation should converge to the discrete Gibbs state
\begin{equation}\label{eq:discrete_gibbs}
\rho_i^{\mathrm{eq}} = \frac{1}{Z_h}\,\exp(-V_i), \quad V_i = V(x_i),
\end{equation}
where $Z_h = h\sum_{j=0}^{N-1} e^{-V_j}$ is the discrete normalization constant. As pointed out in \cite{wang2003robust}, this holds if the jump rates satisfy the \emph{detailed balance condition}
\begin{equation}
\frac{F_{i+\frac12}}{B_{i+\frac12}} = \frac{\rho^{\mathrm{eq}}_{i+1}}{\rho^{\mathrm{eq}}_i}
= \exp\big(-(V_{i+1}-V_i)\big).
\end{equation}

\begin{remark}[Finite difference vs.\ finite volume]\label{rmk:FV-upwind}
Although the final schemes often look the same, many existing structure-preserving schemes for Fokker--Planck equations are based on a finite-volume viewpoint \cite{wang2003robust, heida2021consistency}.
In a finite-volume (FV) scheme, $\rho_i$ is the cell average
\begin{equation}
   \rho_i(t) = \frac{1}{h}\int_{x_i - h/2}^{x_i + h/2} \rho(x,t)\,\mathrm{d}x,
   \qquad p_i = h\,\rho_i,
\end{equation}
so mass conservation holds by construction, but the discrete equilibrium is complicated: physical consistency requires $\rho_i^{\mathrm{eq}}$ to be the cell average of the Gibbs measure,
\begin{equation}
 \rho_i^{\mathrm{eq}} = \frac{1}{h}\int_{I_i} \frac{1}{Z}\,\exp(-V(x))\,\mathrm{d}x.
\end{equation}
If one instead enforces the pointwise Boltzmann relation $\rho_i^{\mathrm{eq}} \propto \exp(-V_i)$, then $V_i$ can no longer equal $V(x_i)$ and must instead be the effective coarse-grained potential
\begin{equation}\label{eq:effective_V}
 V_i^{\mathrm{eff}} = -\ln \left( \frac{1}{h} \int_{I_i} \exp(-V(x)) \,\mathrm{d}x \right).
\end{equation}
This added complexity is why we work with finite differences throughout, taking $\rho_i \approx \rho(x_i, t)$ with the simple discrete equilibrium $\rho_i^{\mathrm{eq}} \propto \exp(-V_i)$. 
\end{remark}

The central difference rates \eqref{eqn: central difference rates} do not, in
general, satisfy the detailed-balance condition, since
\begin{equation}
    \frac{F_{i+\frac12}^{\rm CD}}{B_{i+\frac12}^{\rm CD}}
    =
    \frac{1-\frac12\delta V_{i+\frac12}}
         {1+\frac12\delta V_{i+\frac12}}
    \neq
    \exp\bigl(-\delta V_{i+\frac12}\bigr).
\end{equation}
Indeed, setting $w=\delta V_{i+\frac12}$, a Taylor expansion gives
\[
    \frac{1-\frac12w}{1+\frac12w}
    =
    e^{-w}\left(1-\frac{w^3}{12}+O(w^5)\right).
\]
For a smooth potential,
\[
    \delta V_{i+\frac12}
    =
    hV'(x_{i+\frac12})+O(h^3),
\]
and hence the discrepancy between the two rate ratios is $O(h^3)$. Thus,
although the central difference scheme is not exactly reversible at any fixed
mesh size, its rate ratio is asymptotically consistent with detailed balance
for smooth potentials. However, for discontinuous potentials, $\delta V_{i+1/2}$ can be $O(1)$ and the discrepancy does not vanish under mesh refinement. Consequently, the
central difference scheme may fail to recover the correct equilibrium.
Moreover, when $\lvert\delta V_{i+\frac12}\rvert>2$, one of the transition
rates becomes negative, and positivity of the numerical solution is no longer
guaranteed.

A similar calculation applies to the classical upwind scheme \cite{courant1952solution}. Its transition
rates are
\[
    F_{i+\frac12}^{\rm UP}
    =
    \frac{1}{h^2}
    \left[1+\bigl(-\delta V_{i+\frac12}\bigr)_+\right],
    \qquad
    B_{i+\frac12}^{\rm UP}
    =
    \frac{1}{h^2}
    \left[1+\bigl(\delta V_{i+\frac12}\bigr)_+\right]\ ,
\]
where $(x)_+ = \max(0, x)$.
Setting $w=\delta V_{i+\frac12}$, their ratio is
\[
    \frac{F_{i+\frac12}^{\rm UP}}
         {B_{i+\frac12}^{\rm UP}}
    =
    \begin{cases}
        (1+w)^{-1}, & w\geq 0,\\
        1-w,        & w<0.
    \end{cases}
\]
A piecewise Taylor expansion yields
\[
    \frac{F_{i+\frac12}^{\rm UP}}
         {B_{i+\frac12}^{\rm UP}}
    =
    e^{-w}
    \left(1+\frac12 w|w|+O(|w|^3)\right).
\]
Hence, for a smooth potential, the upwind rate ratio approximates the
detailed-balance ratio with an $O(h^2)$ error. Across a discontinuity
of $V$, however, $w$ can remain $O(1)$, and the discrepancy does not
vanish under mesh refinement. Unlike the central difference rates,
the upwind rates remain positive for arbitrary $w$.

To construct a scheme that satisfies the detailed balance condition, 
in \cite{wang2003robust}, the authors constructed the transition rates from a local stationary Fokker--Planck
problem posed on each pair of neighboring cells from a finite-volume viewpoint. Consider
\(x_{i+\frac12}\), translate it to the origin, and set
\(w=\delta V_{i+\frac12}=V_{i+1}-V_i\), so that the potential gradient is approximated
by the constant \(w/h\) on the two adjacent cells. 
The goal is to reconstruct a local stationary density
\(\widetilde{\rho}\) such that the flux $J:=-\left(\frac{\dd\widetilde{\rho}}{\dd x}+\frac{w}{h}\widetilde{\rho}\right)$ is constant on \((-h,h)\). This leads to the stationary problem
\begin{equation}\label{eq:WPE_local_problem}
\frac{\dd}{\dd x}
\left(
\frac{\dd\widetilde{\rho}}{\dd x}
+\frac{w}{h}\widetilde{\rho}
\right)
=0,
\qquad x\in(-h,h),
\end{equation}
subject to the two mass constraints
\begin{equation}\label{eq:WPE_mass_constraints}
\int_{-h}^{0}\widetilde{\rho}(x)\,\dd x=h\rho_i,
\qquad
\int_{0}^{h}\widetilde{\rho}(x)\,\dd x=h\rho_{i+1}.
\end{equation}
A general solution of (\ref{eq:WPE_local_problem}) is given by 
\(\widetilde{\rho}(x)=c_1e^{-wx/h}+c_2\), where the two constants are determined by \eqref{eq:WPE_mass_constraints}.
Since \(J=-\tfrac{w}{h}c_2\), only \(c_2\) is needed, and
eliminating it gives the flux
\begin{equation}\label{eq:WPE_flux}
J_{i+\frac12}^{\rm WPE}
=
\frac{1}{h}
\left[
\psi_{\mathrm{Bern}}(w)\,\rho_i
-
\psi_{\mathrm{Bern}}(-w)\,\rho_{i+1}
\right],
\qquad
\psi_{\mathrm{Bern}}(w)
:=
\frac{w}{e^w-1},
\quad
\psi_{\mathrm{Bern}}(0):=1.
\end{equation}
Substituting \eqref{eq:WPE_flux} into the discrete conservation law
\(\dd\rho_i/\dd t=(J_{i-\frac12}-J_{i+\frac12})/h\) and reading off the coefficients of
\(\rho_{i\pm1}\) gives the transition rates
\begin{equation}
F_{i+\frac12}^{\rm WPE}
=
\frac{1}{h^2}\,\psi_{\mathrm{Bern}}(\delta V_{i+\frac12})
=
\frac{1}{h^2}
\frac{\delta V_{i+\frac12}}
{e^{\delta V_{i+\frac12}}-1},
\qquad
B_{i+\frac12}^{\rm WPE}
=
\frac{1}{h^2}\,\psi_{\mathrm{Bern}}(-\delta V_{i+\frac12})
=
\frac{1}{h^2}
\frac{\delta V_{i+\frac12}}
{1-e^{-\delta V_{i+\frac12}}}.
\end{equation}
It is easy to verify that these rates satisfy the detailed balance condition at every mesh size and for arbitrarily large
\(\delta V_{i+\frac12}\). Moreover, both rates are positive, since \(\psi_{\mathrm{Bern}}>0\) on
\(\mathbb{R}\).
The exponential profile of the local stationary problem explains the term
exponential fitting. In the linear constant-diffusion setting
considered here, the same Bernoulli-function flux has
appeared, with slightly different derivations, as the Il'in--Allen--Southwell scheme for
convection--diffusion equations
\cite{AllenSouthwell1955,Iln1969}, the Scharfetter--Gummel scheme for
semiconductor transport \cite{SG1969}, and the Chang--Cooper scheme in kinetic theory \cite{CHANG1970}.

An alternative derivation, given in \cite{wang2007convergence}, uses the substitution $\rho(x, t) = g(x, t) \exp(-V(x))$, which transforms \eqref{FP_linear} into a weighted diffusion equation for $g$:
\begin{equation}
 \exp(-V(x))\frac{\partial g}{\partial t} = \nabla \cdot \left( \exp\left(-V(x) \right) \nabla g \right).
\end{equation}
Discretizing this equation directly yields the jump rates
\begin{equation}
F_{i+\frac12} =  \frac{1}{h^2}\exp(- V_{i+\frac12})\exp(V_{i}),
\qquad
B_{i+\frac12} = \frac{1}{h^2}\exp(- V_{i+\frac12}) \exp(V_{i+1}),
\end{equation}
where $V_{i+\frac12}$ is a suitable average of the potential at the two neighboring grid points, e.g.\ $V_{i+\frac12} = (V_i + V_{i+1})/2$. Different choices of the average lead to different schemes \cite{heida2021consistency}. The same
reformulation, often referred to as the Slotboom, or non-logarithmic Landau, transform, is used to construct positivity-preserving
and free-energy-dissipating schemes for nonlinear and nonlocal diffusion
equations \cite{liu2018positivity, liu2020positive, ding2023convergence}.

\section{Variational Construction of Jump Rates}
\label{sec:rates}

In this section, we propose a variational framework for
constructing structure-preserving master equation discretizations of Fokker--Planck equations. Rather than discretizing the differential
operator in the governing PDE, we project the variational
structure of the Fokker--Planck equation onto a finite-dimensional grid function
representation. The projected kinematics gives a conservative
semi-discrete evolution, which can be viewed as a master equation for a Markov jump process.
The projected energy--dissipation law
determines the \emph{net flux} between two states. The detailed-balance condition, which ensures consistency with the correct equilibrium distribution, determines the \emph{ratio} of the forward and backward jump rates along each edge. Together, these two conditions determine the jump rates. Because the finite-dimensional representation contains only values at the grid points, the required edge quantities are obtained through local reconstructions from neighboring grid values, with different reconstruction
choices leading to different schemes.

To illustrate the idea, in this section, we focus on the one-dimensional linear Fokker--Planck equation
\begin{equation}\label{FP_se3}
 \pp_t \rho  = \nabla \cdot (\nabla \rho + \rho \nabla V(x))
\end{equation}
for $x \in (0, 1)$, subject to periodic boundary conditions, where $V(x)$ is a periodic function. As shown previously, \eqref{FP_se3} satisfies the energy-dissipation law \eqref{ED_mu_1}. The framework developed here is general and extends naturally to no-flux boundary conditions, higher-dimensional settings, and more complex physical models; some of these extensions are discussed in Section~\ref{sec:extensions}.

Following the setting of Section~\ref{subsec:trad_disc}, we begin by defining the discrete free energy functional $\mathcal{F}_h[\bm{\rho}]$ of the grid function ${\bm \rho}$, which is a direct quadrature approximation of its continuous counterpart:
\begin{equation}\label{eq:discrete_F}
  \mathcal{F}_h[\bm{\rho}] = h \sum_{i=0}^{N-1} \left( \rho_i \ln \rho_i - \rho_i + \rho_i V_i \right), \qquad V_i = V(x_i).
\end{equation}
If $V$ is smooth and $\rho_i(t)$ is a second-order approximation of $\rho(x_i,t)$, then $\mathcal{F}_h$ is a second-order quadrature approximation of the continuous free energy $\mathcal{F}$.

For the dissipation part, we first approximate $\nabla \mu$ by the discrete gradient on the staggered grid  as
\begin{equation}
  \nabla_h \mu_{i + 1/2} = \frac{\mu_{i+1} - \mu_i}{h},
  \qquad \mu_i := \ln \rho_i + V_i,
\end{equation}
which is a second-order approximation of the continuous gradient $\nabla \mu(x_{i+1/2})$ under the smoothness assumptions on $\rho$ and $V$.
Then the continuous dissipation is approximated by the discrete dissipation
\begin{equation}\label{ap_Diss}
   \triangle_h[\bm{\rho}] = h \sum_{i=0}^{N-1} \widehat{\rho}_{i + 1/2} \, \big| \nabla_h \mu_{i + 1/2} \big|^2,
\end{equation}
where $\widehat{\rho}_{i+1/2}$ is an approximation of $\rho(x_{i+1/2})$, reconstructed from $\bm{\rho}$. To ensure that \(\triangle_h[\bm{\rho}]\) is also a second-order approximation of the continuous dissipation $\int_0^1 \rho\,|\nabla \mu|^2\,\dd x$, we require
\(
\widehat{\rho}_{i+1/2}=\rho(x_{i+1/2})+O(h^2).
\)

\begin{remark}
In this finite-difference setting, the discretizations of the free energy and dissipation are straightforward, because $\rho_i$ is a
point value: $\mathcal{F}_h$ is a quadrature rule applied to nodal values, $\mu_i=\ln\rho_i+V_i$ is
the nodal chemical potential, and $\triangle_h$ combines a midpoint rule with a finite difference approximation. The only quantity that has to be reconstructed is the interfacial density
$\widehat{\rho}_{i+1/2}$. A finite-volume viewpoint is less direct, since $\rho_i$ is a cell average
while the energy density is nonlinear: in general
\begin{equation}\label{eq:FV_jensen_gap}
  \rho_i\ln\rho_i-\rho_i
  \;\neq\;
  \frac1h\int_{I_i}\bigl(\rho\ln\rho-\rho\bigr)\,\dd x ,
\end{equation}
and $\ln\rho_i+V(x_i)$ is not the cell average of $\mu$. Evaluating $\mathcal{F}_h$ and
$\triangle_h$ then requires reconstructing a pointwise density inside each cell first, and
consistency of the discrete equilibrium additionally forces $V(x_i)$ to be replaced by the
coarse-grained potential $V_i^{\mathrm{eff}}$ of \eqref{eq:effective_V}
(Remark~\ref{rmk:FV-upwind}). All of these discrepancies are $O(h^2)$ for smooth data, so the two
viewpoints are equivalent to the order of accuracy considered here. For simplicity, we work with finite differences
so that $\widehat{\rho}_{i+1/2}$ remains the only reconstruction in the construction.
\end{remark}

The main idea is to choose the jump rates $F_{i+1/2}$ and $B_{i+1/2}$ in \eqref{master_rho_sec3} such that the time derivative of the discrete free
energy along the master equation \eqref{master_rho_sec3} matches the discrete
dissipation \eqref{ap_Diss}. Differentiating $\mathcal{F}_h$ and applying a discrete
summation by parts (shifting the index in the backward-flux sum by periodicity)
gives
\begin{equation}\label{dFdt_master}
\frac{\dd \mathcal{F}_h}{\dd t}
= h \sum_{i=0}^{N-1}
\big(\mu_{i+1}-\mu_i\big)\,
\big(F_{i+\frac12}\rho_{i}-B_{i+\frac12}\rho_{i+1}\big).
\end{equation}
Requiring
\eqref{dFdt_master} to equal $-\triangle_h$ \emph{edge by edge} yields
\begin{equation}\label{cond1}
\boxed{
F_{i + \frac12}\rho_{i}-B_{i+\frac12}\rho_{i+1}
= -\frac{1}{h^2}\,\widehat{\rho}_{i+\frac12}
\big(\mu_{i+1}-\mu_i\big).}
\end{equation}

To determine the forward and backward rates separately, an additional structural condition is needed. The natural candidate is the \emph{detailed balance condition},
\begin{equation}\label{cond2}
  \boxed{
F_{i+\frac12}\,e^{-V_{i}} = B_{i+\frac12}\,e^{-V_{i+1}}\ ,}
\end{equation}
which encodes the equilibrium information of the Gibbs state and fixes the \emph{ratio} between $F_{i+1/2}$ and $B_{i+1/2}$. This specifies how the net flux is split into forward and backward components.
The detailed balance condition \eqref{cond2} is equivalent to the existence of a scalar coefficient \(\gamma_{i+\frac12}(\bm{\rho})\) such that
\begin{equation}\label{cond2-2}
F_{i+1/2} = \frac{1}{h^2}\gamma_{i+1/2} (\bm{\rho}) \, e^{V_i}, \qquad
B_{i+1/2} = \frac{1}{h^2}\gamma_{i+1/2} (\bm{\rho}) \, e^{V_{i+1}}.
\end{equation}
Substituting the expressions from \eqref{cond2-2} into the consistency condition \eqref{cond1}, we obtain an explicit formula for \(\gamma_{i+\frac12}(\bm{\rho})\):
\begin{equation}\label{exact_interfaical_gamma}
\gamma_{i+ \frac12}^{\rm exact}(\bm{\rho})
= \widehat{\rho}_{i+\frac12}\,
\frac{\mu_{i+1}-\mu_i}
{e^{V_{i+1}}\rho_{i+1} - e^{V_{i}}\rho_{i}} = \widehat{\rho}_{i+\frac12}\,
\frac{\mu_{i+1}-\mu_i}
{e^{\mu_{i+1}}-e^{\mu_i}}, \quad i = 0, \ldots, N - 1.
\end{equation}
Equation \eqref{exact_interfaical_gamma} determines the edge coefficient exactly, but the resulting rates depend on the density. For the linear Fokker--Planck equation considered here, it is desirable to construct rates that depend only on the prescribed potential, so that the resulting master equation remains linear and its generator can be fixed independently of the evolving density. For a smooth positive density, we observe that
\begin{equation}
 \frac{\mu(x_{i+1})-\mu (x_i)}
{e^{\mu(x_{i+1})}-e^{\mu(x_i)}} = \frac{1}{e^{\mu (x_{i+1/2})}} + O(h^2) = \frac{1}{\rho(x_{i+1/2})} e^{-V(x_{i+1/2})} + O(h^2),
\end{equation}
Since $\widehat{\rho}_{i+1/2}$ approximates
$\rho(x_{i+1/2})$, this observation suggests replacing the exact
density-dependent coefficient in \eqref{exact_interfaical_gamma} by
the density-independent approximation
\begin{equation}\label{approx_interfaical_gamma}
 \gamma_{i+\frac12}^{\rm exact}(\bm{\rho})
 \approx e^{-V_{i+1/2}}
\end{equation}
where $V_{i+1/2}$ is a suitable approximation of $V(x_{i+1/2})$, obtained from a suitable local reconstruction of the potential. Accordingly, the resulting
density-independent rates are
\begin{equation}\label{rate_with_V_half}
 F_{i+1/2}
 =
 \frac{1}{h^2}e^{-(V_{i+1/2}-V_i)},
 \qquad
 B_{i+1/2}
 =
 \frac{1}{h^2}e^{-(V_{i+1/2}-V_{i+1})}.
\end{equation}
Different reconstructions of $V_{i+1/2}$ lead to different practical schemes, as discussed in the next subsection. The approximation \eqref{approx_interfaical_gamma} will then be justified through a consistency analysis.

\begin{remark}
For a more general free energy
\begin{equation}
 \mathcal{F}[\rho] = \int_\Omega \Bigl( \omega(\rho) + \rho V(x) + \tfrac{1}{2} (K * \rho)\, \rho \Bigr) \dd x,
\end{equation}
the same derivation gives
\begin{equation}\label{flux_consistent}
F_{i + \frac12}\rho_{i}-B_{i+\frac12}\rho_{i+1}
= -\frac{1}{h^2} \widehat{\rho}_{i+1/2}
\big(\mu_{i+1} - \mu_i \big).
\end{equation}
Thus, as in the linear Fokker--Planck case, the variational structure determines the net flux across each interface. However, this condition alone does not uniquely determine the two jump rates
\(F_{i+1/2}\) and \(B_{i+1/2}\). One may choose the forward and backward jump rates \(F_{i+1/2}\) and \(B_{i+1/2}\) in many different ways so long as the above identity is satisfied. For example, in \cite{carrillo2015finite, chow2019entropy, BailoCarrilloHu2020}, an upwind-type splitting is used:
\begin{equation}
F_{i+\frac12}({\bm \rho}) = \frac{1}{h^2}\,(\mu_i-\mu_{i+1})_+,
\qquad
B_{i+\frac12}({\bm \rho}) = \frac{1}{h^2}\,(\mu_{i+1}-\mu_i)_+,
\end{equation}
where \((a)_+ = \max\{a,0\}\). If the interface density is chosen by upwinding,
\begin{equation}
\widehat{\rho}_{i+\frac12} =
\begin{cases}
\rho_i, & \mu_i \ge \mu_{i+1},\\
\rho_{i+1}, & \mu_{i+1} > \mu_i,
\end{cases}
\end{equation}
then this splitting is consistent with \eqref{flux_consistent}. Higher-order spatial accuracy can be obtained by reconstructing nonnegative left and right edge values $\rho_{i+\frac12}^{-}$ and $\rho_{i+\frac12}^{+}$ from a wider stencil and selecting the upstream value according to the sign of $\mu_i-\mu_{i+1}$ \cite{BailoCarrilloHu2020}.

This nonlinear upwind splitting preserves equilibria through the
vanishing of the discrete thermodynamic force. Indeed, a positive
discrete equilibrium is characterized by $\mu_i^{\rm eq}=\lambda$
for all $i$, where $\lambda$ is the Lagrange multiplier associated
with mass conservation. Consequently,
\[
F_{i+\frac12}( \bm \rho^{\rm eq})
=
B_{i+\frac12}(\bm \rho^{\rm eq})=0,
\]
and every interface flux vanishes. In the linear Fokker--Planck case,
$\mu_i=\ln\rho_i+V_i$, so this condition gives the discrete Gibbs
state
\[
\rho_i^{\rm eq}=C_h\exp(-V_i).
\]
This mechanism differs from detailed balance for
density-independent transition rates. In a reversible linear master
equation, the forward and backward fluxes are generally nonzero at
equilibrium but cancel pairwise. For the nonlinear upwind splitting,
both density-dependent rates vanish when the chemical-potential
difference vanishes. Thus, preservation of the Gibbs state follows
from discrete force balance rather than from a prescribed ratio of
forward and backward rates.

\end{remark}

A related generalization replaces the entropy $\rho\ln\rho$ by a general convex function of the density ratio against a prescribed equilibrium $\pi$, leading to $f$-divergence energies.

\begin{remark}
The same derivation applies to $f$-divergence energies, where
\begin{equation}
    \frac{\dd}{\dd t} \int_{0}^1 f \left(\frac{\rho(x)}{\pi(x)} \right) \pi(x) \dd x = - \int_0^1 \rho |\nabla \mu|^2 \dd x,
    \qquad \mu:=f'\!\left(\frac{\rho}{\pi}\right) .
\end{equation}
Here the free energy is the $f$-divergence between $\rho$ and the known positive equilibrium $\pi(x)$.
The corresponding generalizations of flux consistency \eqref{cond1} and the detailed balance condition \eqref{cond2} are
\begin{equation}\label{cond1_1}
F_{i + 1/2}\rho_{i}-B_{i+1/2}\rho_{i+1}
= -\frac{1}{h^2} \widehat\rho_{i+1/2} \left(f'\left( \frac{\rho_{i+1}}{\pi_{i+1}}\right) - f'\left(\frac{\rho_i}{\pi_i}\right) \right),
\end{equation}
and
\begin{equation}\label{cond2_1-f}
F_{i+\frac12} \pi_i = B_{i+\frac12} \pi_{i+1}.
\end{equation}
This implies that there exists \(\gamma_{i+1/2}(\bm{\rho})\) such that
\begin{equation}\label{cond2-2-f}
F_{i+1/2} = \frac{1}{h^2}\frac{\gamma_{i+1/2}}{\pi_i}, \qquad
B_{i+1/2} =  \frac{1}{h^2}\frac{\gamma_{i+1/2}} {\pi_{i+1}}.
\end{equation}
Thus, we can derive that
\begin{equation}
  \gamma_{i+1/2} =  \widehat\rho_{i+1/2} \frac{f'\bigl( \frac{\rho_{i+1}}{\pi_{i+1}}\bigr) - f'\bigl(\frac{\rho_i}{\pi_i}\bigr)}{ \frac{\rho_{i+1}}{\pi_{i+1}} - \frac{\rho_i}{\pi_i}} \approx  \widehat\rho_{i+1/2}\, f''\left( \frac{\rho_{i+1/2}}{\pi_{i+1/2}}\right) ,
\end{equation}
 where $\pi_{i+1/2}$ is a suitable approximation of $\pi(x_{i+1/2})$.
\end{remark}

\subsection{Representative reconstructions}
\label{subsec:reconstructions}

To construct a practical scheme, we need to determine how to reconstruct $V_{i+1/2}$, $e^{-V_{i+1/2}}$, or $e^{V_{i+1/2}}$ from the nodal values $V_i,V_{i+1}$. Different second-order reconstructions recover classical exponential-fitting schemes and produce additional detailed-balance rates. Throughout this subsection, let $\delta V_{i+1/2}=V_{i+1}-V_i$.

Some typical second-order reconstructions, including those underlying several well-known schemes in the literature, are summarized below.  

\begin{itemize}

\item \textbf{Elston--Doering (ED) rates.} The midpoint approximation $V_{i+1/2}\approx(V_i+V_{i+1})/2$ gives the symmetric rates used in \cite{elston1996}:
\begin{equation}
F_{i+1/2}^{\text{ED}} = \frac{1}{h^2} e^{-\frac{\delta V_{i+1/2}}{2}}, \quad B_{i+1/2}^{\text{ED}} = \frac{1}{h^2} e^{\frac{\delta V_{i+1/2}}{2}}.
\end{equation}
Such a reconstruction is equivalent to applying the geometric mean, also known as the square-root approximation (SQRA) \cite{donati2021markov, heida2021consistency}, to $e^{-V}$, i.e., 
\begin{equation}
e^{-V_{i+1/2}} \approx \gamma_{i+1/2}^{\rm ED} :=  G(e^{-V_i}, e^{-V_{i+1}}) = \sqrt{e^{-V_i} e^{-V_{i+1}}}.
\end{equation}

\item \textbf{SG/WPE rates.} The classical Scharfetter--Gummel \cite{SG1969} or Wang--Peskin--Elston \cite{wang2003robust} rate can be recovered by applying the logarithmic mean to $e^{V_{i+1/2}}$:
\begin{equation}
e^{V_{i+1/2}} \approx L(e^{V_{i}}, e^{V_{i+1}}) = \begin{cases}
     \dfrac{e^{V_{i+1}} - e^{V_i}}{V_{i+1} - V_i}, &  V_{i+1} \neq V_{i}, \\[2mm]
     e^{V_{i}}, & V_{i+1} = V_i,
 \end{cases}
 \end{equation}
and therefore
\begin{equation}
\gamma^{\text{SG}}_{i +1/2} =  \frac{V_{i+1}-V_i}{e^{V_{i+1}} - e^{V_i}} = \frac{V_{i+1}-V_i}{e^{V_{i+1}}(1 - e^{-(V_{i+1}-V_i)})}, \quad V_{i+1} \neq V_i.
\end{equation}
Hence
\begin{equation}
\begin{aligned}
F_{i+1/2}^{\rm SG} =\frac{1}{h^2}\,\frac{\delta V_{i+1/2}}{e^{\delta V_{i+1/2}}-1}, ~~B_{i+1/2}^{\rm SG}  = \frac{1}{h^2}\,\frac{\delta V_{i+1/2}}{1-e^{-\delta V_{i+1/2}}}\ .
\end{aligned}
\end{equation}

\item  \textbf{Improved WPE (IWPE) rates.} The harmonic mean reconstruction
\begin{equation}
e^{-V_{i+1/2}} \approx \gamma_{i+1/2}^{\rm IWPE} := H(e^{-V_i}, e^{-V_{i+1}}) =  \frac{2}{e^{V_i} + e^{V_{i+1}}}
\end{equation}
leads to the rates proposed in \cite{wang2008convergence}:
\begin{equation}
F_{i+1/2}^{\rm IWPE}=\frac{2}{h^2}\,\frac{1}{1+e^{\delta V_{i+1/2}}},
\qquad
B_{i+1/2}^{\rm IWPE}=\frac{2}{h^2}\,\frac{1}{1+e^{-\delta V_{i+1/2}}}.
\end{equation}
The IWPE rate is well suited to discontinuous potentials \cite{wang2007convergence} (see the numerical experiments in Section~\ref{subsec:exp2}). 

  \item \textbf{Arithmetic mean of $e^{-V}$ (AM-$e^{-V}$).} 
   Use the arithmetic mean for $e^{-V_{i+1/2}}$:
   \begin{equation}
   e^{-V_{i+1/2}} \approx \gamma_{i+1/2}^{\rm AM} :=  A(e^{-V_i}, e^{-V_{i+1}}) = \frac{e^{-V_i} + e^{-V_{i+1}}}{2}.
   \end{equation}
    This gives
   \begin{equation}\label{eq:TypeI_rates}
   F_{i+1/2}^{\text{AM}} = \frac{1}{h^2} \frac{1 + e^{-\delta V_{i+1/2}}}{2}, \quad B_{i+1/2}^{\text{AM}} = \frac{1}{h^2} \frac{e^{\delta V_{i+1/2}} + 1}{2}.\end{equation}

\item  \textbf{Logarithmic mean of $e^{-V}$ (LM-$e^{-V}$).}
Apply the logarithmic mean to $e^{-V_{i+1/2}}$ rather than to $e^{V_{i+1/2}}$:
\begin{equation}
e^{-V_{i+1/2}} \approx \gamma_{i+1/2}^{\rm LM} := L(e^{-V_i}, e^{-V_{i+1}}) = - \frac{e^{-V_{i+1}} - e^{-V_i}}{V_{i+1} - V_i}.
\end{equation}
Then
\begin{equation}\label{eq:TypeII_rates}
    F_{i+1/2}^{\text{LM}} = \frac{1}{h^2} \frac{1 - e^{-\delta V_{i+1/2}}}{\delta V_{i+1/2}}, \quad B_{i+1/2}^{\text{LM}} = \frac{1}{h^2} \frac{e^{\delta V_{i+1/2}} - 1}{\delta V_{i+1/2}}.
\end{equation}
% Again, the $0/0$ cancellation at $\delta V_{i+1/2}=0$ is removable and should be evaluated stably.

\end{itemize}

The removable singularities in the SG/WPE and LM-$e^{-V}$ formulas at
$\delta V_{i+\frac12}=0$ are understood by continuity and should be
evaluated using numerically stable implementations.

\begin{remark}
All rates above can be written in the form
\begin{equation}\label{eq:psi_rates}
    F_{i+\frac12}
    =
    \frac{1}{h^2}\psi(\delta V_{i+\frac12}),
    \qquad
    B_{i+\frac12}
    =
    \frac{1}{h^2}\psi(-\delta V_{i+\frac12}),
\end{equation}
where \(\psi:\mathbb{R}\to(0,\infty)\) and \(\psi(0)=1\). This is closely
related to the \(B\)-scheme notation used in the finite-volume literature
\cite{Hillairet2011,li2020large}. Since the discrete Gibbs state satisfies
\(\rho^{\mathrm{eq}}_{i+1}/\rho^{\mathrm{eq}}_i=e^{-\delta V_{i+\frac12}}\), the detailed-balance condition $F_{i+\frac12}\rho^{\mathrm{eq}}_i
    =
    B_{i+\frac12}\rho^{\mathrm{eq}}_{i+1}$
is equivalent to
\begin{equation}\label{eq:multiplicative_symmetry}
    \ln\psi(-w)-\ln\psi(w)=w.
\end{equation}
This multiplicative symmetry is used in the modified \(B\)-scheme and
Stolarsky-mean formulations of
\cite{heida2021consistency,li2025discrete}. It guarantees that the sampled
Gibbs state is an exact stationary state of the discrete system. By contrast, the classical \(B\)-scheme framework imposes the additive
condition
\begin{equation}\label{eq:additive_B_symmetry}
    \psi(w)-\psi(-w)=-w,
\end{equation}
which is used to recover the correct drift contribution in the finite-volume flux \cite{Hillairet2011,li2020large}.  The two conditions
therefore encode different properties: \eqref{eq:multiplicative_symmetry}
fixes the discrete equilibrium through detailed balance, whereas
\eqref{eq:additive_B_symmetry} is a consistency condition inherited from the
drift--diffusion flux. Requiring both conditions uniquely determines
\[
    \psi(w)=\frac{w}{e^w-1},
\]
namely, the Bernoulli function associated with the SG/WPE scheme. The SG/WPE flux reduces to the upwind flux in the vanishing-diffusion limit \cite{hraivoronska2023}.

Several analytical results are already available for these related classes
of schemes. Convergence of classical \(B\)-schemes for stationary
finite-volume problems was established in \cite{Hillairet2011}. For
one-dimensional Fokker--Planck equations on the full line,
\cite{li2020large} interpreted \(B\)-schemes as jump
processes and proved a discrete Poincar\'e inequality, exponential
convergence to the discrete stationary state, and an \(O(h)\)
uniform-in-time error estimate under suitable confinement and regularity
assumptions. Consistency and energy-norm convergence estimates for the
Stolarsky-mean family, which includes the five reconstructions considered
here, were developed in \cite{heida2021consistency}. More recently,
\cite{li2025discrete} established discrete Poincar\'e inequalities, with
constants uniform in \(h\), and exponential convergence to the exact
discrete Gibbs state for modified \(B\)-schemes satisfying
\eqref{eq:multiplicative_symmetry}. These results provide convergence and
long-time analyses for substantial parts of the linear fixed-potential
setting. 
\end{remark}

\subsection{Variational Consistency}\label{subsec:stability_consistency}

In this subsection, we formally justify the density-independent approximation \eqref{approx_interfaical_gamma} of the exact coefficient \eqref{exact_interfaical_gamma} and analyze the five proposed reconstructions by establishing the consistency of the discrete energy and dissipation functionals with their continuous counterparts.

When the rates are defined through \eqref{cond2-2}, the exact coefficient
$\gamma_{i+\frac12}^{\rm exact}$ and each reconstructed coefficient
$\gamma_{i+\frac12}^{\rm rec}$, with ${\rm rec}\in
\bigl\{\text{\rm ED},\ \text{\rm SG/WPE},\ \text{\rm IWPE},
\ \text{\rm AM},\ \text{\rm LM}\bigr\},$
give rise to rates satisfying the detailed-balance condition \eqref{cond2}.
Consequently, each associated master equation \eqref{master_rho_sec3}
satisfies the energy--dissipation law derived in
Section~\ref{subsec:master_var}.

\begin{proposition}
\label{prop:EDD_consistency}

Assume periodic indexing, and let the rates be defined by the
detailed-balance parametrization \eqref{cond2-2} with a nonnegative edge
coefficient $\gamma_{i+\frac12}$. This includes the exact coefficient
$\gamma_{i+\frac12}^{\rm exact}(\bm\rho)$ and each of the reconstructed
coefficients $\gamma_{i+\frac12}^{\rm rec}$. If $\rho_i(t)>0$ and
$\bm\rho(t)$ solves \eqref{master_rho_sec3}, then
\begin{equation}\label{eq:ME_EDD_identity}
  \frac{\dd}{\dd t}\mathcal F_h[\bm\rho(t)]
  =
  -\triangle_h^{\rm ME}[\bm\rho(t)]
  \leq 0,
\end{equation}
where the dissipation associated with the chosen edge coefficient is
\begin{equation}\label{eq:ME_dissipation}
  \triangle_h^{\rm ME}[\bm\rho]
  :=
  h\sum_{i=0}^{N-1}
  \gamma_{i+\frac12}
  \frac{e^{\mu_{i+1}}-e^{\mu_i}}{h}
  \frac{\mu_{i+1}-\mu_i}{h} \geq 0
\end{equation}
For the exact coefficient
$\gamma_{i+\frac12}=\gamma_{i+\frac12}^{\rm exact}(\bm\rho)$ defined in
\eqref{exact_interfaical_gamma}, 
\begin{equation}\label{eq:exact_D_match}
  \triangle_h^{\rm ME}[\bm\rho]
  =
  \triangle_h[\bm\rho],
\end{equation}
where $\triangle_h[\bm\rho]$, defined in \eqref{ap_Diss}, is the finite-difference approximation of the continuous Fokker--Planck dissipation.
\end{proposition}

\begin{remark}\label{rmk:DB-continuum-EDL}

Proposition~\ref{prop:EDD_consistency} also indicates a formal continuum
limit of the detailed-balance master equation \eqref{master_rho_sec3}.
Suppose that the discrete solutions converge smoothly as $h\to0$ and that
\[
  \gamma_{i+\frac12}\longrightarrow\gamma(x)
  \qquad\text{as}\qquad
  x_{i+\frac12}\to x.
\]
Formally, one then has
\begin{equation}
  \mathcal F_h\longrightarrow\mathcal F,
  \qquad
  \triangle_h^{\rm ME}
  \longrightarrow
  \int_\Omega
  \gamma(x)\,\partial_x(e^\mu)\,\partial_x\mu\,\dd x
  =
  \int_\Omega
  \gamma(x)e^\mu|\partial_x\mu|^2\,\dd x.
\end{equation}
To recover the original Fokker--Planck equation, one must have $\gamma(x)e^\mu=\rho,$
or equivalently,
\begin{equation}
  \gamma(x)=e^{-V(x)}.
\end{equation}
This formal continuum-limit argument motivates the approximation
\eqref{approx_interfaical_gamma}; the five reconstructions provide different
finite-mesh approximations of the same limiting coefficient.

\end{remark}

To establish variational consistency, we show in the remainder of this subsection that, for each fixed $t$ and any smooth strictly positive density $\rho(\x,t)$ with discrete representation $P_h\rho(t):=\bigl(\rho(x_i,t)\bigr)_{i=0}^{N-1},$
the dissipation induced by any of the five reconstructed coefficients $\gamma_{i+\frac12}^{\rm rec}$ satisfies
\begin{equation}
 \triangle_h^{\rm ME}[P_h\rho(t)]
 =
 \triangle_h[P_h\rho(t)]
 +O(h^2),
\end{equation}
under suitable smoothness assumptions on $V$. Indeed, we only need to show the second-order consistency between $\gamma^{\rm rec}_{i+\frac12}$ and $\gamma^{\rm exact}_{i+\frac12}(P_h\rho)$, which is proved in the following lemma.

\begin{lemma}
\label{prop:second_order_reconstruction}

Assume that $V$ has a $1$-periodic $C^{1,1}$ extension, and define
\begin{equation}
  \widehat{\gamma}_{i+\frac12}
  :=
  e^{-V(x_{i+\frac12})}.
\end{equation}
For each $\gamma^{\rm rec}_{i+\frac12}$, we have
\begin{equation}\label{eq:gamma_rec_consistency}
  \gamma^{\rm rec}_{i+\frac12}
  =
  \widehat{\gamma}_{i+\frac12}
  +O(h^2)
\end{equation}
uniformly in $i$. Let, in addition, $\rho$ be a positive periodic $C^2$ function, define its
discrete representation by $P_h\rho
  :=
  \bigl(\rho(x_i)\bigr)_{i=0}^{N-1},$
and assume that $\widehat\rho_{i+\frac12}
  =
  \rho(x_{i+\frac12})+O(h^2)$
uniformly in $i$. Then the exact density-dependent coefficient satisfies
\begin{equation}\label{eq:gamma_exact_consistency}
  \gamma^{\rm exact}_{i+\frac12}(P_h\rho)
  =
  \widehat{\gamma}_{i+\frac12}
  +O(h^2)
\end{equation}
uniformly in $i$. Consequently,
\begin{equation}\label{eq:gamma_rec_vs_exact}
    \gamma^{\rm exact}_{i+\frac12}(P_h\rho) = \gamma^{\rm rec}_{i+\frac12} +O(h^2).
\end{equation}

\end{lemma}

\begin{proof}
Let $x_m=x_{i+\frac12}$. For a positive $C^{1,1}$ function $u$, set
\[
  u_i:=u(x_i),\qquad
  a:=\frac{u_i+u_{i+1}}{2},
  \qquad
  d:=\frac{u_{i+1}-u_i}{2}.
\]
A symmetric expansion about $x_m$ gives
\[
  a=u(x_m)+O(h^2),
  \qquad
  d=O(h),
\]
uniformly in $i$. Moreover, the arithmetic, geometric, harmonic, and
logarithmic means satisfy
\[
  M(a-d,a+d)=a+O(d^2).
\]
Applying this observation to $u=e^{-V}$ proves
\eqref{eq:gamma_rec_consistency} for the ED, IWPE, AM, and LM
reconstructions. For the SG/WPE reconstruction,
\[
  \gamma^{\rm SG}_{i+\frac12}
  =
  \frac{1}{L(e^{V_i},e^{V_{i+1}})}
  =
  e^{-V(x_m)}+O(h^2)
  =
  \widehat{\gamma}_{i+\frac12}+O(h^2),
\]
where
\[
  L(a,b):=\frac{b-a}{\ln b-\ln a},
  \qquad
  L(a,a):=a.
\]
This proves \eqref{eq:gamma_rec_consistency} for all five
reconstructions.

For the exact coefficient, set
\[
  g:=e^V\rho=e^\mu.
\]
Equation \eqref{exact_interfaical_gamma} can then be written as
\[
  \gamma^{\rm exact}_{i+\frac12}(P_h\rho)
  =
  \frac{\widehat\rho_{i+\frac12}}
       {L(g_i,g_{i+1})}.
\]
Since $g$ is positive and $C^{1,1}$,
\[
  L(g_i,g_{i+1})
  =
  g(x_m)+O(h^2).
\]
Using
\[
  \widehat\rho_{i+\frac12}
  =
  \rho(x_m)+O(h^2),
  \qquad
  g(x_m)=e^{V(x_m)}\rho(x_m),
\]
we obtain
\[
  \gamma^{\rm exact}_{i+\frac12}(P_h\rho)
  =
  e^{-V(x_m)}+O(h^2)
  =
  \widehat{\gamma}_{i+\frac12}+O(h^2),
\]
which proves \eqref{eq:gamma_exact_consistency}. Equation
\eqref{eq:gamma_rec_vs_exact} follows immediately from
\eqref{eq:gamma_rec_consistency} and
\eqref{eq:gamma_exact_consistency}.
\end{proof}

Now we are ready to establish the variational consistency of the semi-discrete master equation with the continuous Fokker--Planck equation for the five reconstructed rates.

\begin{proposition}
\label{lemma:functional_consistency}

Let $\rho$ be a positive periodic function satisfying
\[
  V,\rho\in C_{\rm per}^2([0,1]),
  \qquad
  \mu=\ln\rho+V\in C_{\rm per}^3([0,1]),
  \qquad
  \rho\geq\rho_*>0,
\]
and define its discrete representation by $P_h\rho:=\bigl(\rho(x_i)\bigr)_{i=0}^{N-1}.$
Assume also that $\widehat\rho_{i+\frac12} = \rho(x_{i+\frac12})+O(h^2)$
uniformly in $i$. Then
\begin{equation}\label{eq:F_consistency}
  \mathcal F_h[P_h\rho]
  =
  \mathcal F[\rho]+O(h^2),
  \qquad
  \triangle_h[P_h\rho]
  =
  \triangle[\rho]+O(h^2).
\end{equation}
Moreover, for any of the five reconstructed coefficients
$\gamma_{i+\frac12}^{\rm rec}$,
\begin{equation}\label{eq:ME_D_consistency}
  \triangle_h^{\rm ME}[P_h\rho]
  =
  \triangle_h[P_h\rho]+O(h^2)
  =
  \triangle[\rho]+O(h^2),
\end{equation}
where $\triangle_h^{\rm ME}$ is the dissipation induced by
$\gamma_{i+\frac12}^{\rm rec}$.

\end{proposition}

\begin{proof}
The consistency of $\mathcal F_h$ follows from the composite trapezoidal
rule. A symmetric expansion about $x_{i+\frac12}$ gives
\[
  \frac{\mu_{i+1}-\mu_i}{h}
  =
  \partial_x\mu(x_{i+\frac12})+O(h^2).
\]
Together with the assumed approximation of
$\widehat\rho_{i+\frac12}$ and the composite midpoint rule, this yields
\[
  \triangle_h[P_h\rho]
  =
  h\sum_{i=0}^{N-1}
  \rho(x_{i+\frac12})
  |\partial_x\mu(x_{i+\frac12})|^2
  +O(h^2)
  =
  \triangle[\rho]+O(h^2).
\]

For a reconstructed coefficient,
\eqref{eq:exact_D_match} and \eqref{eq:gamma_rec_vs_exact} give
\begin{align*}
  \triangle_h^{\rm ME}[P_h\rho]
  -\triangle_h[P_h\rho]
  =
  h\sum_{i=0}^{N-1}
  \left(
    \gamma_{i+\frac12}^{\rm rec}
    -
    \gamma_{i+\frac12}^{\rm exact}(P_h\rho)
  \right)
  \frac{e^{\mu_{i+1}}-e^{\mu_i}}{h}
  \frac{\mu_{i+1}-\mu_i}{h} =O(h^2),
\end{align*}
since both difference quotients are uniformly bounded. Combining the last
two estimates proves \eqref{eq:ME_D_consistency}.
\end{proof}

\section{Temporal Discretizations}\label{sec:temporal}
In this section, we discuss time discretizations for the semi-discrete equation 
\begin{equation}\label{eq:semi_disc_linear}
\frac{\dd}{\dd t}\rho_i
= B_{i+\frac12}\rho_{i+1} - F_{i+\frac12}\rho_i
  - B_{i-\frac12}\rho_i + F_{i-\frac12}\rho_{i-1}, \quad  i = 0, 1, \ldots, N-1,
\end{equation}
with periodic indexing, i.e., $\rho_{-1}=\rho_{N-1}$, $\rho_N=\rho_0$, $B_{-\frac12}=B_{N-\frac12}$, and $F_{-\frac12}=F_{N-\frac12}$.
For simplicity, we only discuss the periodic boundary condition throughout this section. No-flux boundary conditions can be treated analogously by setting the boundary fluxes to zero. In matrix form, the semi-discrete system \eqref{eq:semi_disc_linear} can be written as
\begin{equation}
\frac{\dd}{\dd t}\,\bm\rho = A_h\,\bm\rho, \qquad
\bm\rho = (\rho_0,\rho_1,\dots,\rho_{N-1})^{\top}.
\end{equation}
Here, $A_h$ is a conservative Metzler matrix, i.e.\ the off-diagonals are nonnegative, diagonal entries are nonpositive, and each column sums to zero.

We first consider the fully explicit forward Euler scheme. Applying this to the semi-discrete equation \eqref{master_rho_sec3} yields:
\begin{equation}
\frac{\rho_i^{n+1} - \rho_i^n}{\Delta t} = B_{i+1/2} \rho_{i+1}^n - F_{i+1/2} \rho_i^n - B_{i-1/2} \rho_i^n + F_{i-1/2} \rho_{i-1}^n.
\end{equation}
Rearranging the terms, we obtain the update formula:
\begin{equation}
\rho_i^{n+1} = \Delta t B_{i+1/2} \rho_{i+1}^n + \left(1 - \Delta t (F_{i+1/2} + B_{i-1/2}) \right) \rho_i^n + \Delta t F_{i-1/2} \rho_{i-1}^n.
\end{equation}
This scheme can be interpreted as a discrete-time Markov chain, where $p_i^n = h\rho_i^n$ represents the probability mass at state $i$ at time $t_n$. To ensure that $\rho_i^{n+1}$ remains nonnegative given $\rho^n \ge 0$, the coefficients must be nonnegative. This imposes the following CFL-type condition on the time step $\Delta t$:
\begin{equation}
\Delta t \le \frac{1}{\max_i (F_{i+1/2} + B_{i-1/2})}.
\end{equation}
For a constant potential $V(x) = \text{const}$, this reduces to the classical heat-equation condition $\Delta t \le h^2/2$.

To overcome the strict time-step restriction, we employ the implicit Euler discretization. In matrix notation, the scheme reads:
\begin{equation}\label{Im_Euler}
\frac{\bm{\rho}^{n+1} - \bm{\rho}^n}{\Delta t} = A_h \bm{\rho}^{n+1} \quad \implies \quad (I - \Delta t A_h) \bm{\rho}^{n+1} = \bm{\rho}^n,
\end{equation}
where $I$ is the identity matrix. Since the matrix $(I - \Delta t A_h)$ is invertible for any $\Delta t > 0$, the implicit Euler scheme is uniquely solvable. Moreover, unconditional energy stability, positivity preservation, and mass conservation can be established, as shown in Proposition~\ref{prop:implicit_euler} below.

\begin{proposition}\label{prop:implicit_euler}

Let $A_h$ be the fixed, time-independent Markov generator associated with the
jump rates. Thus, $(A_h)_{ij}\geq 0 \quad (i\neq j)$ and $\bm{1}^{\top}A_h=\bm{0}^{\top}.$
Assume, in addition, that $A_h$ satisfies detailed balance with respect to the
discrete Gibbs state $\bm{\rho}^{\mathrm{eq}}$ of \eqref{eq:discrete_gibbs}, where
$\rho_i^{\mathrm{eq}}=Z_h^{-1}e^{-V_i}>0$, namely, $(A_h)_{ij}\rho_j^{\mathrm{eq}}=(A_h)_{ji}\rho_i^{\mathrm{eq}}, ~i\neq j.$
For a positive state $\bm{\rho}$, let
$\bm{\mu}=\bm{\mu}(\bm{\rho})$ denote the discrete chemical potential defined
componentwise by $\mu_i=\ln\rho_i+V_i$. Then, for any $\Delta t>0$, the
implicit Euler scheme (\ref{Im_Euler}) satisfies the following properties:

\begin{enumerate}
    \item \textbf{Unconditional positivity preservation.}
    If $\bm{\rho}^{n} > 0$, then $\bm{\rho}^{n+1} > 0$.

    \item \textbf{Mass conservation.}
    $\sum_i\rho_i^{n+1}=\sum_i\rho_i^n$.

    \item \textbf{Energy stability.}
    The discrete free energy is non-increasing:
    $\mathcal{F}_h[\bm{\rho}^{n+1}]
    \leq\mathcal{F}_h[\bm{\rho}^{n}]$.
\end{enumerate}

\end{proposition}

\begin{proof}

\textit{Positivity and mass conservation.}
Let $M=I-\Delta t A_h$. Since $(A_h)_{ij}\geq 0$ for $i\neq j$, the off-diagonal
entries of $M$ are nonpositive. Moreover, the zero-column-sum condition gives
$(A_h)_{jj}=-\sum_{i\neq j}(A_h)_{ij}$, and hence
\[
M_{jj}
=
1+\Delta t\sum_{i\neq j}(A_h)_{ij}
>
\Delta t\sum_{i\neq j}(A_h)_{ij}
=
\sum_{i\neq j}|M_{ij}|.
\]
Thus, $M$ is a strictly column-diagonally-dominant nonsingular $M$-matrix, so
$M^{-1}$ is entrywise nonnegative. It follows that
$\bm{\rho}^{n+1}=M^{-1}\bm{\rho}^{n} > 0$ whenever
$\bm{\rho}^{n} > 0$. Multiplying \eqref{Im_Euler} from the left by $\bm{1}^{\top}$ and
using $\bm{1}^{\top}A_h=\bm{0}^{\top}$ gives
$\bm{1}^{\top}\bm{\rho}^{n+1}
=\bm{1}^{\top}\bm{\rho}^{n}$, which proves mass conservation.

\noindent \textit{Energy stability of implicit Euler.}
To prove the energy stability, we first show that
$\langle\bm{\mu},A_h\bm{\rho}\rangle_h\leq 0$, where
$\langle\bm{u},\bm{v}\rangle_h
:=h\sum_i u_iv_i$. Define $r_i=\rho_i/\rho_i^{\mathrm{eq}}$. For $i\neq j$, detailed balance
allows us to introduce the symmetric coefficients
$c_{ij}:=(A_h)_{ij}\rho_j^{\mathrm{eq}}=(A_h)_{ji}\rho_i^{\mathrm{eq}}\geq 0$. Using the zero-column-sum condition $\bigl(\sum_{i=1}^N (A_h)_{ij}\bigr)\rho_j = 0$,
we obtain
\[
(A_h\bm{\rho})_i
=
\sum_{j\neq i}
\left((A_h)_{ij}\rho_j-(A_h)_{ji}\rho_i\right)
=
\sum_{j\neq i}c_{ij}(r_j-r_i).
\]
Since $\rho_i^{\mathrm{eq}}=Z_h^{-1}e^{-V_i}$, the chemical potential satisfies
$\mu_i=\ln r_i-\ln Z_h$. The constant $-\ln Z_h$ makes no contribution
because $\bm{1}^{\top}A_h=\bm{0}^{\top}$. Therefore,
\begin{align}
\langle\bm{\mu},A_h\bm{\rho}\rangle_h
=
h\sum_i\ln r_i\,(A_h\bm{\rho})_i \notag =
-\frac{h}{2}
\sum_i\sum_{j\neq i}
c_{ij}(r_j-r_i)(\ln r_j-\ln r_i)
\leq 0,                                           \label{eq:generator_dissipation}
\end{align}
because the logarithm is increasing. The discrete free energy $\mathcal{F}_h$ is convex, and its variational
derivative with respect to the discrete inner product is
$\bm{\mu}(\bm{\rho})$. The convexity inequality, applied at
$\bm{\rho}^{n+1}$, gives
\[
\mathcal{F}_h[\bm{\rho}^{n}]
-
\mathcal{F}_h[\bm{\rho}^{n+1}]
\geq
\left\langle
\bm{\mu}^{n+1},
\bm{\rho}^{n}-\bm{\rho}^{n+1}
\right\rangle_h.
\]
Using
$\bm{\rho}^{n}-\bm{\rho}^{n+1}
=-\Delta t A_h\bm{\rho}^{n+1}$ and
\eqref{Im_Euler}, we obtain
\begin{equation}\label{eq:BE_energy}
\mathcal{F}_h[\bm{\rho}^{n}]
-
\mathcal{F}_h[\bm{\rho}^{n+1}]
\geq
-\Delta t
\left\langle
\bm{\mu}^{n+1},
A_h\bm{\rho}^{n+1}
\right\rangle_h
\geq 0.
\end{equation}
This proves the claimed energy stability.

\end{proof}

Since $A_h$ is time independent, the Crank--Nicolson scheme can also be applied,
yielding second-order accuracy in time \cite{wang2007convergence}. It preserves
mass, but positivity and free-energy dissipation are generally not
unconditional and may require an additional time-step restriction; we
therefore do not discuss it further.

\section{Extensions to More General Settings}\label{sec:extensions}
In this section, we briefly discuss extensions of the proposed framework to more general settings. The construction of Section~\ref{sec:rates} uses only two ingredients: the discrete energy--dissipation law, which fixes the \emph{net flux} across an interface, and detailed balance, which fixes the \emph{splitting} of that flux into a forward and a backward rate. So the construction extends naturally to more general mobilities and free energies, and to higher dimensions.

% ---------------------------------------------------------------
\subsection{Nonlinear Mobility}\label{subsec:nonlinear_mob}
% ---------------------------------------------------------------

We first consider a FP equation of the form (\ref{FP_general}), where $m(\rho,\x)\geq0$ is the nonlinear mobility and 
$
\mathcal F[\rho]=\int_\Omega(\rho\ln\rho-\rho+\rho V)\,\dd x
$ is the free energy.  Under periodic or no-flux boundary
conditions, the equation (\ref{FP_general}) satisfies the energy--dissipation law
\begin{equation}\label{ED_nonlinear_M}
  \rho_t=\nabla\cdot\bigl(m(\rho,\x)\nabla\mu\bigr),
  \qquad
  \frac{\dd}{\dd t}\mathcal F[\rho]
  =-\int_\Omega m(\rho,\x)|\nabla\mu|^2\,\dd x\leq0.
\end{equation}
On a one-dimensional uniform grid, let $m_{i+\frac12}\geq0$ be a
second-order edge reconstruction of the mobility.  Replacing
$\widehat\rho_{i+\frac12}$ by $m_{i+\frac12}$ in
\eqref{ap_Diss} gives
\begin{equation}\label{Dh_nonlinear_M}
 \triangle_h[\bm\rho]
 =h\sum_{i=0}^{N-1}m_{i+\frac12}
   |\nabla_h\mu_{i+\frac12}|^2.
\end{equation}
The edgewise energy--dissipation matching and detailed balance then give
$F_{i+\frac12}=\frac{1}{h^2}\gamma_{i+\frac12}e^{V_i}$,
$B_{i+\frac12}=\frac{1}{h^2}\gamma_{i+\frac12}e^{V_{i+1}}$, with the exact
coefficient
\begin{equation}\label{gamma_nonlinear_M}
 \gamma^{\mathrm{exact}}_{i+\frac12}(\bm\rho)
 =m_{i+\frac12}
  \frac{\mu_{i+1}-\mu_i}
       {e^{\mu_{i+1}}-e^{\mu_i}}
 =\frac{m_{i+\frac12}}
       {L(e^{\mu_i},e^{\mu_{i+1}})}.
\end{equation}
When $\mu_{i+1}=\mu_i$, the quotient is understood by continuity as
$e^{-\mu_i}$.  Hence $\gamma^{\mathrm{exact}}_{i+\frac12}\geq0$. The only difference is that we now also need to specify the reconstruction of $\frac{m_{i+1/2}}{ \widehat\rho_{i+\frac12} }$ for a general mobility. Let $a(\rho,x):=m(\rho,x)/\rho$ for $\rho>0$, assuming a nonnegative
model-dependent extension to $\rho=0$ whenever degenerate states are
admitted, and set
\begin{equation}\label{eq:kinetic_reconstruction}
  \widehat a_{i+\frac12}(\bm\rho)
  :=\mathcal A\bigl(a(\rho_i,x_i),a(\rho_{i+1},x_{i+1})\bigr)\geq0,
\end{equation}
where $\mathcal A$ is a symmetric consistent mean.  If $\rho$ is
smooth and bounded away from zero and both $\widehat a_{i+\frac12}$
and $V_{i+\frac12}$ are second-order midpoint reconstructions, then the analysis in Lemma \ref{prop:second_order_reconstruction} implies that
$\gamma^{\mathrm{exact}}_{i+\frac12}$ and
$\widehat a_{i+\frac12}e^{-V_{i+\frac12}}$ have the same
leading-order expansion.  This motivates the computable rates
\begin{equation}\label{rates_nonlinear_M_product}
  F_{i+\frac12}(\bm\rho)
  =\frac{\widehat a_{i+\frac12} (\bm\rho)}{h^2}
    e^{-(V_{i+\frac12}-V_i)},
  \qquad
  B_{i+\frac12}(\bm\rho)
  =\frac{\widehat a_{i+\frac12} (\bm\rho) }{h^2}
    e^{-(V_{i+\frac12}-V_{i+1})},
\end{equation}
where $V_{i+\frac12}=V(x_{i+\frac12})+O(h^2)$.  These rates preserve
detailed balance exactly; the choice of $\mathcal A$ controls the
approximation of the nonlinear mobility.

Because the nonlinear mobility enters the two directional rates through the
common edge coefficient $\gamma_{i+\frac12}(\bm\rho)$, we evaluate this
coefficient explicitly at time $t^n$ while treating the density implicitly:
\begin{equation}\label{SI_nonlinear_M}
  \bigl(I-\Delta t\,A_h(\bm\rho^n)\bigr)\bm\rho^{n+1}
  =
  \bm\rho^n.
\end{equation}
This is a frozen-coefficient semi-implicit discretization. For every frozen state
$\bm\rho^n$ for which $\gamma_{i+\frac12}^n\geq0$, the resulting rates satisfy
\[
  F_{i+\frac12}^n e^{-V_i}
  =
  B_{i+\frac12}^n e^{-V_{i+1}},
\]
and hence $A_h(\bm\rho^n)$ is a conservative Metzler generator satisfying
detailed balance with respect to the same Gibbs state
$\rho_i^{\mathrm{eq}}\propto e^{-V_i}$. If some edge coefficients vanish, detailed balance
still holds, although the frozen chain may become reducible and the Gibbs
state need not be its unique invariant state.

\begin{proposition}
\label{prop:nonlinear_M}
Assume that $\bm\rho^n > 0$ componentwise  and that
$\gamma_{i+\frac12}^n\geq0$ for every edge. Then, for every $\Delta t>0$,
\eqref{SI_nonlinear_M} has a unique solution satisfying
\[
  \bm\rho^{n+1} > 0,
  \qquad
  h\sum_i\rho_i^{n+1}
  =
  h\sum_i\rho_i^n,
  \qquad
  \mathcal F_h[\bm\rho^{n+1}]
  \leq
  \mathcal F_h[\bm\rho^n].
\]
Moreover, The discrete free energy is non-increasing: $\mathcal F_h[\bm\rho^{n+1}] \leq \mathcal F_h[\bm\rho^n]$.
\end{proposition}

\begin{proof}
For each fixed $n$, $A_h^n:=A_h(\bm\rho^n)$ satisfies all the assumptions of
Proposition~\ref{prop:implicit_euler}. Therefore, existence, uniqueness,
nonnegativity, mass conservation, and free-energy decay follow directly from
that proposition. Strictly positive data remain strictly positive because
the diagonal entries of $(I-\Delta t\,A_h^n)^{-1}$ are positive.

\end{proof}

% ---------------------------------------------------------------
\subsection{Nonlocal Interaction Energies}\label{subsec:nonlocal}
% ---------------------------------------------------------------

In this section, we discuss how to extend the above construction to cases in which the free energy contains a nonlocal interaction term \cite{carrillo2015finite,chow2019entropy}. We consider the free energy
\begin{equation}\label{F_nonlocal}
  \mathcal{F}[\rho]
  = \int_\Omega \bigl(\rho\ln\rho-\rho+\rho V(x)\bigr)\,\dd x
  + \frac12\int_\Omega\!\int_\Omega K(x,x')\,\rho(x)\,\rho(x')\,\dd x\,\dd x'.
\end{equation}
Here, $K(x,x')=K(x',x)$ is a symmetric kernel, and the chemical potential is
\begin{equation}\label{eq:nonlocal_chemical_potential}
  \mu(\rho,x) = \ln\rho(x)+\widetilde V(\rho,x),
  \qquad
  \widetilde V(\rho,x) := V(x) + \int_\Omega K(x,x')\,\rho(x')\,\dd x' .
\end{equation}
A positive equilibrium with unit mass is a constrained critical point of $\mathcal F$ and hence satisfies $\ln\rho^{\mathrm{eq}}+\widetilde V(\rho^{\mathrm{eq}},\cdot)=\lambda$ for a multiplier $\lambda$, or equivalently the fixed-point relation
\begin{equation}\label{eq:self_consistent_fixed_point}
  \rho^{\mathrm{eq}}(x)
  = \frac{\exp\bigl(-\widetilde V(\rho^{\mathrm{eq}},x)\bigr)}
         {\displaystyle\int_\Omega \exp\bigl(-\widetilde V(\rho^{\mathrm{eq}},y)\bigr)\,\dd y} .
\end{equation}
The existence, uniqueness, and bifurcation structure of such self-consistent equilibria depend on the particular forms of $V$ and $K$, and multiple nonhomogeneous equilibria may occur
\cite{carrillo2020long, carrillo2026long}.
Since the equilibrium $\rho^{\mathrm{eq}}$ is in general not available in closed form, it cannot serve a priori as the reference state in the detailed-balance condition \eqref{cond2}.

To overcome this challenge, we freeze the self-consistent potential over each time step. This approach is similar to the schemes in \cite{liu2018positivity, schlichting2022scharfetter, ding2023convergence}. At the continuous level, given $\rho^n$ we set $\widetilde V^n(x):=\widetilde V(\rho^n,x)=V(x)+\int_\Omega K(x,x')\rho^n(x')\,\dd x'$ and approximate the nonlocal dynamics on $(t^n,t^{n+1})$ by the linear Fokker--Planck equation
\begin{equation}\label{FP_local_to_nonlocal}
  \pp_t \rho = \nabla\cdot\bigl(\rho\nabla \widetilde V^n+\nabla\rho\bigr),
  \qquad t\in(t^{n},t^{n+1}),
\end{equation}
to which the construction of Section~\ref{sec:rates} applies.  In the discrete setting, given $\bm\rho^n$, put
\begin{equation}\label{eq:lagged_potential}
  \widetilde V_i^n = V_i + h\sum_{j=0}^{N-1} K(x_i,x_j)\,\rho_j^n,
  \qquad i=0,\ldots,N-1,
\end{equation}
that is, $\widetilde{\bm V}^n=\bm V+\mathcal K_h\bm\rho^n$, where
$K_{ij}:=K(x_i,x_j)$ and $(\mathcal K_h\bm u)_i:=h\sum_{j=0}^{N-1}K_{ij}u_j$
is the discrete interaction operator, the trapezoidal quadrature of
$u\mapsto\int_\Omega K(\cdot,y)u(y)\,\dd y$. Writing
$\langle\bm u,\bm v\rangle_h:=h\sum_i u_iv_i$ and $\|\bm u\|_h^2:=\langle\bm u,\bm u\rangle_h$,
the \emph{true} discrete free energy corresponding to \eqref{F_nonlocal} is
\begin{equation}\label{eq:F_h_nonlocal}
  \mathcal F_h[\bm\rho]
  = h\sum_{i=0}^{N-1}\bigl(\rho_i\ln\rho_i-\rho_i+V_i\rho_i\bigr)
    + \tfrac12\bigl\langle\mathcal K_h\bm\rho,\bm\rho\bigr\rangle_h ,
\end{equation}
which reduces to \eqref{eq:discrete_F} when $K\equiv0$.
We can assemble the rates from any of the reconstructions of Section~\ref{sec:rates} applied to $\widetilde V^n$,
\begin{equation}\label{eq:lagged_rates}
  F_{i+\frac12}^n = \frac{1}{h^2}\gamma_{i+\frac12}^n\,e^{\widetilde V_i^n},
  \qquad
  B_{i+\frac12}^n = \frac{1}{h^2}\gamma_{i+\frac12}^n\,e^{\widetilde V_{i+1}^n} .
\end{equation}
Hence, the detailed balance holds with respect to the \emph{lagged Gibbs state} $\rho_i^{\mathrm{eq},n}\propto e^{-\widetilde V_i^n}$. Writing $A_h^n=A_h(\bm\rho^n)$ for the generator assembled from \eqref{eq:lagged_rates}, backward Euler gives the semi-implicit update
\begin{equation}\label{eq:SI_update}
  \bigl(I-\Delta t\,A_h^n\bigr)\bm\rho^{n+1}=\bm\rho^n .
\end{equation}
The nonlocal coupling enters only through the explicitly evaluated potential $\widetilde{\bm V}^n$, so each step costs one sparse linear solve together with one discrete convolution, evaluated by FFT for a translation-invariant kernel on a periodic grid. Since $I-\Delta t\,A_h^n$ is an $M$-matrix, nonnegativity and mass are preserved for every $\Delta t>0$.

\noindent {\bf Fixed points are self-consistent equilibria.}
If the update stalls, $\bm\rho^{n+1}=\bm\rho^n=\bm\rho^\ast$, then \eqref{eq:SI_update} gives $A_h(\bm\rho^\ast)\bm\rho^\ast=0$, which for an irreducible frozen generator is equivalent to
\begin{equation}\label{eq:discrete_self_consistent_equilibrium}
  \rho_i^\ast \propto \exp\bigl(-\widetilde V_i(\bm\rho^\ast)\bigr),
\end{equation}
the discrete counterpart of \eqref{eq:self_consistent_fixed_point}. Thus, if the discrete dynamics converges, its limit is a self-consistent discrete
equilibrium. General convergence to equilibrium is not addressed here and may remain
open even for the continuous PDE, particularly for irregular interaction
kernels \cite{carrillo2026well}.

\noindent {\bf Accuracy.} We give only a formal accuracy analysis of the lagged approach. Assume
that \(V\), \(K\), and the solution are smooth and periodic, with the density
bounded away from zero. Backward Euler is first-order accurate in time, and
lagging the nonlocal potential introduces an additional \(O(\Delta t)\) error, so
the expected temporal order remains one. For each frozen potential,
\cref{prop:second_order_reconstruction} and
\cref{lemma:functional_consistency} give second-order spatial
consistency of the rate construction. The periodic trapezoidal rule used for
the nonlocal potential is also at least second-order accurate, and is
spectrally accurate for smooth periodic data, so it does not reduce the
spatial order. We therefore formally expect the lagged scheme to be first-order in time and second-order in space.
A proof would require a stability estimate uniform in \(h\), which is not
established here. The predicted second-order spatial rate is observed in
Section~\ref{subsec:exp4} when \(\Delta t=h^2\).

\noindent{\bf Energy stability.} We now discuss the energy stability of the scheme. For each frozen step, \(A_h^n\) satisfies detailed balance with respect to the
lagged Gibbs state. Therefore, as shown in
\cref{prop:nonlinear_M}, the update dissipates the lagged free energy
\begin{equation}\label{eq:frozen_nonlocal_energy}
  \mathcal F_{h,\mathrm{lag}}^n[\bm\rho]
  =
  h\sum_i
  \bigl(
    \rho_i\ln\rho_i-\rho_i+\rho_i\widetilde V_i^n
  \bigr)\ ,
\end{equation}
which does not immediately imply decay of the true discrete free
energy \(\mathcal F_h\) of \eqref{eq:F_h_nonlocal}, as the interaction potential is evaluated at $\bm\rho^n$.  The gap between the two energy increments is an exact quadratic form in $\delta\bm\rho:=\bm\rho^{n+1}-\bm\rho^n$. We have the following proposition.
\begin{proposition}
\label{prop:decomposition}
Assume $K_{ij}=K_{ji}$ and $\bm\rho^n>0$ componentwise, and let
$\bm\rho^{n+1}$ solve \eqref{eq:SI_update} with $\widetilde{\bm V}^n$ given
by \eqref{eq:lagged_potential}; then $\bm\rho^{n+1}>0$. Define
\[
  \mathcal I_h^n
  := -\bigl\langle \ln\bm\rho^{n+1}+\widetilde{\bm V}^n,\;
     A_h^n\bm\rho^{n+1}\bigr\rangle_h ,
  \qquad
  \lambda_{h, +}
  := \Bigl(\sup\Bigl\{
     \tfrac{\langle\mathcal K_h\bm u,\bm u\rangle_h}{\|\bm u\|_h^2}
     \;:\; \bm u\neq\bm 0,\ \textstyle\sum_i u_i=0
     \Bigr\}\Bigr)_{\!+},
\]
where \(r_+:=\max\{r,0\}\), and $\mathcal I_h^n\ge0$ by the detailed-balance argument of
\cref{prop:implicit_euler}. Then
\begin{equation}\label{eq:F_decomp}
  \mathcal F_h[\bm\rho^{n+1}]-\mathcal F_h[\bm\rho^n]
  = \mathcal F_{h,\mathrm{lag}}^n[\bm\rho^{n+1}]
    -\mathcal F_{h,\mathrm{lag}}^n[\bm\rho^n]
    +\mathcal R^n
  \;\le\; -\Delta t\,\mathcal I_h^n+\mathcal R^n\ ,
\end{equation}
where $\mathcal R^n := \tfrac12\bigl\langle\mathcal K_h\,\delta\bm\rho, \delta\bm\rho\bigr\rangle_h$. In particular:
\begin{enumerate}
\item If $\lambda_{h, +}=0$, i.e., $\mathcal K_h$ is negative semidefinite on
  the zero-mass subspace, then
  $\mathcal F_h[\bm\rho^{n+1}]-\mathcal F_h[\bm\rho^n]
   \le-\Delta t\,\mathcal I_h^n\le0$ for every $\Delta t>0$.
\item For a general symmetric kernel, if
  \begin{equation}\label{eq:nonlocal_energy_condition}
    \lambda_{h, +}\,\|\delta\bm\rho\|_h^2\;\le\;\Delta t\,\mathcal I_h^n,
  \end{equation}
  then
  $\mathcal F_h[\bm\rho^{n+1}]-\mathcal F_h[\bm\rho^n]
   \le-\tfrac{\Delta t}{2}\,\mathcal I_h^n\le0$.
\end{enumerate}
\end{proposition}

\begin{proof}
Positivity of $\bm\rho^{n+1}$ follows from \cref{prop:implicit_euler}:
$(I-\Delta t\,A_h^n)^{-1}$ is the inverse of a nonsingular M-matrix and has
positive diagonal entries. The two energies differ only in their
interaction terms,
$\tfrac12\langle\mathcal K_h\bm\rho,\bm\rho\rangle_h$ versus
$\langle\mathcal K_h\bm\rho^n,\bm\rho\rangle_h$; subtracting the two
increments and using the self-adjointness of $\mathcal K_h$ with respect
to $\langle\cdot,\cdot\rangle_h$ gives
\[
  \tfrac12\langle\mathcal K_h\bm\rho^{n+1},\bm\rho^{n+1}\rangle_h
  -\tfrac12\langle\mathcal K_h\bm\rho^{n},\bm\rho^{n}\rangle_h
  -\langle\mathcal K_h\bm\rho^{n},\delta\bm\rho\rangle_h
  = \tfrac12\langle\mathcal K_h\,\delta\bm\rho,\delta\bm\rho\rangle_h
  = \mathcal R^n,
\]
which proves the identity in \eqref{eq:F_decomp}. Since $A_h^n$ satisfies
detailed balance with respect to $\bm\rho^{\mathrm{eq},n}$, the convexity argument of
\cref{prop:implicit_euler}, applied to the frozen generator with potential
$\widetilde{\bm V}^n$, yields
$\mathcal F_{h,\mathrm{lag}}^n[\bm\rho^{n+1}]
 -\mathcal F_{h,\mathrm{lag}}^n[\bm\rho^n]\le-\Delta t\,\mathcal I_h^n$,
which gives the inequality in \eqref{eq:F_decomp}. Since $\mathbf 1^{\!\top}A_h^n=\mathbf 0^{\!\top}$, mass is conserved and
$\sum_i\delta\rho_i=0$. Hence
$\mathcal R^n\le\tfrac12\lambda_{h, +}\|\delta\bm\rho\|_h^2$, which is
nonpositive if $\lambda_{h, +}=0$, and is bounded by
$\tfrac{\Delta t}{2}\mathcal I_h^n$ under
\eqref{eq:nonlocal_energy_condition}. Both conclusions then follow from
\eqref{eq:F_decomp}.
\end{proof}

For a translation-invariant kernel on the periodic grid, $\mathcal K_h$ is
diagonalized by the discrete Fourier transform, and case~(1) applies
whenever the nonconstant discrete Fourier coefficients of the kernel are
nonpositive, as for the periodized attractive Gaussian used in
\cref{subsec:exp4}. For a general continuous kernel, $\lambda_{h, +}$ is
the largest Rayleigh quotient, over mean-zero grid functions, of the
quadrature discretization of the integral operator
$u\mapsto\int_\Omega K(\cdot,y)\,u(y)\,\mathrm dy$, and is therefore
bounded uniformly in $h$; condition
\eqref{eq:nonlocal_energy_condition} thus constrains the time step through
the kernel strength. This is an a posteriori sufficient condition for energy decay.

\begin{remark}
The two cases above are reflected in related semi-implicit schemes. For the
Poisson--Nernst--Planck system, the electrostatic interaction is generated by
a positive-semidefinite Poisson Green operator. The semi-implicit Slotboom
scheme analyzed in \cite{ding2023convergence} evaluates the electrostatic
potential at the previous time level. Its energy estimate contains a
nonnegative quadratic remainder, and energy decay is obtained under a
time-step condition involving bounds on the ionic concentrations and the
electric potential. By contrast, in the parabolic--elliptic Keller--Segel scheme of \cite{liu2018positivity}, the chemoattractant entering the density update is generated by the density at the previous time level. Since the attractive interaction energy has the form $-\frac12\langle\bm\rho,G_h\bm\rho\rangle_h,~ G_h\succeq0,$
this lagged treatment is equivalent to a convex splitting of the interaction
energy, and the corresponding quadratic remainder is nonpositive. These two
examples are consistent with the sign distinction in
\cref{prop:decomposition}.
\end{remark}

\iffalse
\begin{remark}[Comparison with nonlinear upwind schemes]
The upwind rates $F_{i+\frac12}=h^{-2}(\mu_i-\mu_{i+1})_+$, $B_{i+\frac12}=h^{-2}(\mu_{i+1}-\mu_i)_+$ of \cite{carrillo2015finite} depend on $\rho$ both through $\mu$ and through the upwind interface density, so an implicit step requires a nonlinear solve while an explicit step is subject to $\Delta t=O(h^2)$; the scheme is moreover first-order in space. Confining the nonlinearity to $\widetilde V^n$ keeps the step linear and second-order in space.
\end{remark}
\fi

\subsection{Higher-dimensional cases}
\label{subsec:high_dim_edgewise}
% ---------------------------------------------------------------

The one-dimensional construction extends directly to higher dimensions,
because \eqref{cond1} and \eqref{cond2} involve only two neighboring grid
points and can therefore be imposed along each coordinate direction.

Let $\Omega=\prod_{k=1}^{d}(-L_k,L_k)$
be the computational domain, and let $\bm{x}_{\bm{i}} = (-L_1+i_1h_1,\ldots,-L_d+i_dh_d)$, $\bm{i}=(i_1,\ldots,i_d)$, $h_k=2L_k/N_k$,
denote the Cartesian grid points, where ${\bm i} \in \mathcal V_h:=\bigl\{\bm i=(i_1,\ldots,i_d):\ 0\leq i_k\leq N_k-1,\ k=1,\ldots,d\bigr\},$. We use $\bm{e}_k$ to denote the unit
multi-index in the $k$th coordinate direction and set $|C_h|=\prod_{k=1}^{d}h_k.$
As in Section~\ref{sec:rates}, we write
\begin{equation}\label{eq:multidim_notation}
    \rho_{\bm{i}}
    \approx
    \rho(\bm{x}_{\bm{i}}),
    \qquad
    V_{\bm{i}}
    =
    V(\bm{x}_{\bm{i}}),
    \qquad
    \mu_{\bm{i}}
    =
    \ln\rho_{\bm{i}}+V_{\bm{i}},
    \qquad
    g_{\bm{i}}
    =
    e^{\mu_{\bm{i}}}
    =
    e^{V_{\bm{i}}}\rho_{\bm{i}}.
\end{equation}
The discrete free energy is defined by
\begin{equation}\label{eq:multidim_energy}
    \mathcal{F}_h[\bm{\rho}]
    =
    |C_h|
    \sum_{\bm{i}}
    \left[
        \rho_{\bm{i}}\ln\rho_{\bm{i}}
        -
        \rho_{\bm{i}}
        +
        V_{\bm{i}}\rho_{\bm{i}}
    \right].
\end{equation}
For each coordinate edge joining $\bm{i}$ and
$\bm{i}+\bm{e}_k$, let
$m_{\bm{i}+\frac12\bm{e}_k}\geq 0$ be the corresponding discrete mobility.
The discrete dissipation is
\begin{equation}\label{eq:multidim_dissipation}
    \triangle_h[\bm{\rho}]
    =
    |C_h|
    \sum_{\bm{i}}
    \sum_{k=1}^{d}
    m_{\bm{i}+\frac12\bm{e}_k}
    \left(
        \frac{
            \mu_{\bm{i}+\bm{e}_k}
            -
            \mu_{\bm{i}}
        }{h_k}
    \right)^2,
\end{equation}
where each coordinate edge is counted once.

The semi-discrete master equation can be written in the form
\begin{equation}\label{eq:multidim_master}
    \frac{\dd\rho_{\bm{i}}}{\dd t}
    =
    \sum_{k=1}^{d}
    \left(
        J_{\bm{i}-\frac12\bm{e}_k}
        -
        J_{\bm{i}+\frac12\bm{e}_k}
    \right), \quad     J_{\bm{i}+\frac12\bm{e}_k}
    =
    F_{\bm{i}+\frac12\bm{e}_k}\rho_{\bm{i}}
    -
    B_{\bm{i}+\frac12\bm{e}_k}\rho_{\bm{i}+\bm{e}_k},
\end{equation}
with periodic indexing for periodic boundary conditions and vanishing
boundary fluxes for no-flux boundary conditions. Since \eqref{cond1} and \eqref{cond2} relate only the two endpoints of a
single edge, they can be applied to each direction with $h$ replaced by $h_k$:
\begin{equation}\label{eq:multidim_conditions}
    J_{\bm{i}+\frac12\bm{e}_k}
    =
    -\frac{m_{\bm{i}+\frac12\bm{e}_k}}{h_k^2}
    \left(\mu_{\bm{i}+\bm{e}_k}-\mu_{\bm{i}}\right),
    \qquad
    F_{\bm{i}+\frac12\bm{e}_k}e^{-V_{\bm{i}}}
    =
    B_{\bm{i}+\frac12\bm{e}_k}e^{-V_{\bm{i}+\bm{e}_k}}.
\end{equation}
The derivation of Section~\ref{sec:rates} then applies edge by edge and gives
\begin{equation}\label{eq:multidim_rates}
    F_{\bm{i}+\frac12\bm{e}_k}
    =
    \frac{\gamma_{\bm{i}+\frac12\bm{e}_k}}{h_k^2}e^{V_{\bm{i}}},
    \qquad
    B_{\bm{i}+\frac12\bm{e}_k}
    =
    \frac{\gamma_{\bm{i}+\frac12\bm{e}_k}}{h_k^2}e^{V_{\bm{i}+\bm{e}_k}},
    \qquad
    {\gamma}_{\bm{i}+\frac12\bm{e}_k}
    =
    m_{\bm{i}+\frac12\bm{e}_k}
    \frac{\mu_{\bm{i}+\bm{e}_k}-\mu_{\bm{i}}}
         {g_{\bm{i}+\bm{e}_k}-g_{\bm{i}}},
\end{equation}
the quotient being understood by continuity when
$g_{\bm{i}+\bm{e}_k}=g_{\bm{i}}$. For practical schemes, any of the potential reconstructions introduced in Section~3.1 can be combined with the nonlinear-mobility reconstruction of Section~5.1 and applied edge by edge in each coordinate direction.

The structural properties established in one dimension extend edgewise to
the multidimensional scheme. Under periodic or no-flux boundary conditions,
the master-equation form preserves mass and nonnegativity. Moreover, detailed
balance along each coordinate edge determines the discrete Gibbs state $\rho_{\bm i}^{\mathrm{eq}}=\frac{e^{-V_{\bm i}}}{Z_h}$ with $Z_h=|C_h|\sum_{\bm i\in\mathcal V_h}e^{-V_{\bm i}}$, and yields the discrete energy--dissipation law
\begin{equation}
  \frac{\dd}{\dd t}\mathcal F_h[\bm\rho]
  =
  -\triangle_h^{\rm ME}[\bm\rho]
  \leq0,
\end{equation}
where $\triangle_h^{\rm ME}$ denotes the dissipation induced by the chosen
transition rates; see Proposition~\ref{prop:EDD_consistency}.
For the exact density-dependent coefficients in
\eqref{eq:multidim_rates}, the prescribed finite-difference dissipation is
reproduced edge by edge: $\triangle_h^{\rm ME}[\bm\rho]
  =
  \triangle_h[\bm\rho].$
For transition rates constructed from any of the five second-order edge
reconstructions of Section~\ref{subsec:reconstructions}, let $\rho$ be a
smooth strictly positive density and let $P_h\rho$ denote its discrete
representation. Under the multidimensional analogues of the regularity and
interface-reconstruction assumptions in
Lemma~\ref{prop:second_order_reconstruction}, one has
\begin{equation}
  \triangle_h^{\rm ME}[P_h\rho]
  =
  \triangle_h[P_h\rho]+O(h^2),
  \qquad
  h:=\max_{1\leq k\leq d}h_k.
\end{equation}
Consequently, whenever
$\triangle_h[P_h\rho]=\triangle[\rho]+O(h^2)$, we have $ \triangle_h^{\rm ME}[P_h\rho]  = \triangle[\rho]+O(h^2).$
The proof follows by applying the one-dimensional argument of
Proposition~\ref{lemma:functional_consistency} in each coordinate direction and summing over all directions, and is therefore omitted.

At the fully discrete level, the directional generators may either be
assembled into a single multidimensional generator
$A_h=\sum_{k=1}^d A_{h,k}$, followed by one backward-Euler solve, or treated
successively by dimensional splitting. The first approach directly inherits
the mass conservation, nonnegativity, free-energy dissipation, and
equilibrium preservation established above. The splitting approach is more
attractive in high dimensions because each directional substep decomposes
into independent one-dimensional systems, thereby avoiding a coupled
multidimensional solve. To retain the same structural properties, each
directional substep must itself be discretized in a mass-conservative,
positivity-preserving, and free-energy-dissipative manner and must preserve
the same discrete Gibbs state. The resulting composition then preserves these properties, although it introduces a temporal splitting error. Moreover, an ordered composition is generally not reversible even when
every directional substep is reversible; reversibility of the full
transition operator requires either commuting directional operators or a
palindromic composition. Related dimensional-splitting strategies for
nonlinear and nonlocal aggregation--diffusion equations were developed in
\cite{BailoCarrilloHu2020}.

\begin{remark}
The edgewise construction extends naturally to unstructured meshes and
general graphs. Once the neighboring pairs, node or control-volume weights,
and edge transmissibilities have been specified, for instance, following the setup in \cite{zeng2024analysis}, the discrete
energy--dissipation relation and detailed balance can be imposed on each
connection exactly as above. Thus, the construction is algebraically well
defined on any weighted graph. The challenge lies in interpreting it as a discretization of a continuous Fokker--Planck
equation, which requires the resulting graph operator to approximate the
underlying differential operator. This depends on how the node weights and
edge transmissibilities encode the geometry of the domain. For regular admissible finite-volume meshes, convergence of the corresponding discrete gradient-flow structures
was established in \cite{forkert2022evolutionary}. For a general graph,
identifying a continuum differential operator, when one exists, is a separate
graph-to-continuum problem.

In the linear fixed-potential setting, the resulting jump rates are
independent of the evolving density and define a reversible continuous-time Markov chain whose invariant distribution is the Gibbs state. Simulating individual trajectories of this chain therefore provides a natural
route to structure-preserving MCMC methods \cite{li2025discrete}, which will be explored in future work.
\end{remark}

\section{Numerical Experiments}\label{sec:numerics}

In this section, we test the five detailed-balance (DB) schemes introduced in Section~\ref{subsec:reconstructions}: SG/WPE, ED, IWPE, AM-$e^{-V}$, and LM-$e^{-V}$. Unless stated otherwise, the one-dimensional experiments use the periodic nodal grid
$x_i=ih$, $i=0,\ldots,N-1$, with $h=L/N$. Errors are reported in the discrete norms
\begin{equation}\label{eq:numerical_norms}
 \|e\|_{\ell^1}=h\sum_i|e_i|,\qquad
 \|e\|_{\ell^2}=\left(h\sum_i|e_i|^2\right)^{1/2},\qquad
 \|e\|_{\ell^\infty}=\max_i|e_i|,
\end{equation}
with $h_xh_y$ replacing $h$ in two dimensions. For all numerical experiments, we use the implicit Euler scheme or the semi-implicit schemes described in Section~\ref{sec:extensions}. Throughout this section, when we write $\Delta t=f(h)$, we mean the largest constant time step $\Delta t\leq f(h)$ that divides the final time $T$.

%-----------------------------------------------------------------------
\subsection{Smooth potential case}\label{subsec:exp1}
%-----------------------------------------------------------------------

We begin with a linear Fokker--Planck equation with a smooth potential
\begin{equation}\label{eq:V_smooth}
 V(x)=\sin(2\pi x)-\tfrac12\sin(4\pi x)
      +\tfrac13\sin(6\pi x),
\end{equation}
and consider the initial condition
\begin{equation}\label{eq:rho0_smooth}
 \rho_0(x)=Z_0^{-1}e^{-50(x-0.3)^2},
\end{equation}
where $Z_0$ normalizes the mass.

We first test transient accuracy by solving the Fokker--Planck equation until $T=0.05$ using each scheme on grids with
\(
N\in\{8,16,32,64,128,256\}
\)
and $\Delta t=h^2$.  Because no exact transient solution is available, we compute a reference solution using the AM-$e^{-V}$ scheme with $N_{\rm ref}=2048$ and
$\Delta t=h_{\rm ref}^2=2.38\times10^{-7}$.  The nodal grids are nested, so the reference is
restricted by exact sampling at coincident nodes. 
Fig.~\ref{fig:exp1_convergence} shows that all five DB schemes converge at second order in all three norms. The central difference and upwind schemes are second- and first-order accurate, respectively, as expected. Since all three norms yield the same observed order, we report only the $\ell^1$ error in the following experiments.

\begin{figure}[!htbp]
 \centering
 \includegraphics[width=\textwidth]{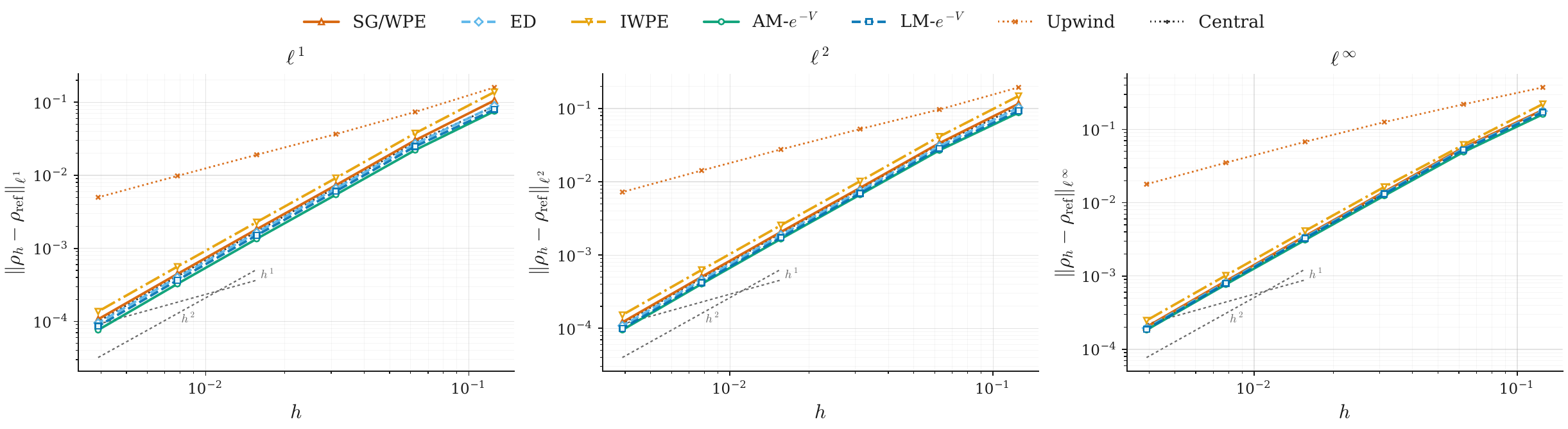}
 \caption{Spatial errors at $T=0.05$ for the smooth potential, measured
 against an AM-$e^{-V}$ reference at $N_{\rm ref}=2048$ restricted to coincident nodes, with
 $\Delta t=h^2$ on every grid.}
 \label{fig:exp1_convergence}
\end{figure}

Next, we consider the long-time behavior of the schemes.  We use $N=256$, $T=1.5$, and
$\Delta t = 10^{-5}$.  The reference is the continuum Gibbs state
$\pi=e^{-V}/Z$, where $Z$ is computed by adaptive quadrature with tolerance $10^{-13}$. Fig.~\ref{eq_smooth}(a)--(b) show that all schemes agree visually with the Gibbs profile very well and dissipate the discrete free energy. As shown in Fig.~\ref{eq_smooth}(c), the DB schemes approximate the Gibbs state more accurately than the upwind and central difference schemes. At $T=1.5$, the DB errors are at
round-off, while the upwind and central difference errors are $1.18\times10^{-2}$ and
$1.12\times10^{-4}$, respectively.  Although the central difference and all DB schemes are both second-order in space for the transient dynamics, only DB schemes reproduce the discrete Gibbs state $\pi_{h,i}\propto e^{-V_i}$ exactly. Fig.~\ref{eq_smooth}(d) further shows the equilibrium error as a function of $h$, using $N\in\{16,32,64,128,256\}$ with $T=1.5$ and $\Delta t=10^{-5}$. Because the DB schemes converge to the exact discrete Gibbs state, the only error comes from the quadrature error between $Z$ and $Z_h$. The DB curves
show no algebraic order because $Z_h$ is the periodic trapezoidal rule applied to a smooth
integrand and is therefore spectrally accurate.  The equilibrium error decreases from
$1.52\times10^{-6}$ at $N=16$ to about $10^{-12}$ at $N=32$ and then remains at round-off. In contrast, the upwind and central-difference schemes exhibit first- and second-order equilibrium errors, respectively, consistent with the analysis of Section~\ref{subsec:trad_disc}. The upwind and central difference rate ratios approximate the exact detailed-balance ratio with local errors of order $O(h^2)$ and $O(h^3)$, respectively, for smooth potentials, and one power is lost when the local errors accumulate over the $O(h^{-1})$ edges of the grid.

\begin{figure}[!h]
  \includegraphics[width = \textwidth]{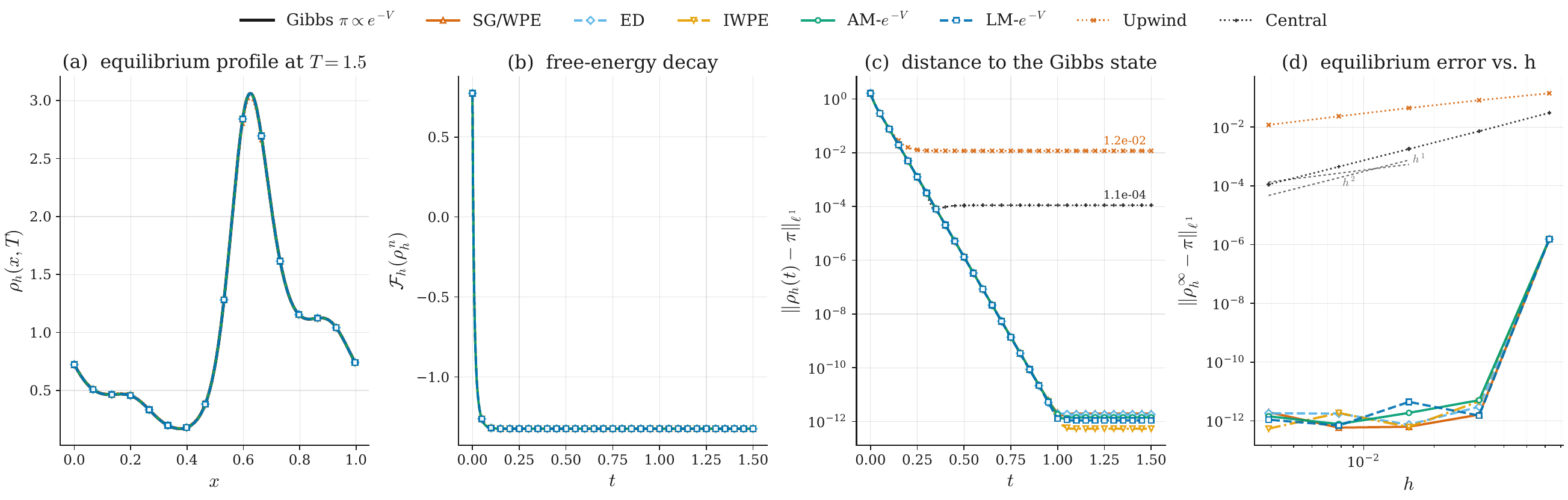}
  \caption{Long-time behavior of all seven schemes at $N=256$, $T=1.5$,
  $\Delta t=10^{-5}$.  (a) Numerical solutions at $T=1.5$ compared with the continuum Gibbs
  state.  (b) Discrete free energy as a function of time.  (c) Distance to the continuum
  Gibbs state in the $\ell^1$ norm as a function of time.  (d) The same equilibrium error
  refined in $h$ over $N\in\{16,\ldots,256\}$.}\label{eq_smooth}
\end{figure}

%-----------------------------------------------------------------------
\subsection{Discontinuous potential}\label{subsec:exp2}
%-----------------------------------------------------------------------

We next consider a potential with a jump at $x=1/2$:
\begin{equation}\label{eq:V_disc}
V(x)=
\begin{cases}
3-6\sin\!\bigl(\tfrac\pi2(x+\tfrac12)\bigr),&x<\tfrac12,\\
3-6\sin\!\bigl(\tfrac\pi2(x-\tfrac12)\bigr),&x\ge\tfrac12,
\end{cases}
\qquad
\rho_0(x)=Z_0^{-1}\max\{1+\cos(2\pi x),0\},
\end{equation}
so that $\delta V=6$ across the interface.  This test is motivated by the examples in
\cite{wang2003robust,wang2007convergence}. Following \cite{wang2007convergence}, we use the cell-centered grid
$x_i=(i+\tfrac12)h$.  For every even $N$, the jump lies at the interface between
$x_{N/2-1}$ and $x_{N/2}$, so all grids resolve the same geometry. 

As shown in \cite{wang2003robust, wang2007convergence}, the central difference scheme can be unstable for discontinuous potentials, producing negative densities and numerical blow-up. For this example, running the central difference scheme with $N=128$ and $\Delta t=h^2$ until $T=0.1$ gives $\max_i|\rho_i|= 2.124\times10^2$ and $\min_i\rho_i= -1.031\times10^2$.  We therefore compare only the five DB schemes and the upwind scheme below.

Fig.~\ref{fig:exp2_profiles}(a) shows the solution at
$t=0.001,0.005,0.01,0.05$ on $N=128$ against the IWPE reference solution at
$N_{\rm ref}=2048$.  We choose $\Delta t=h^2$ and $\Delta t_{\rm ref}=h_{\rm ref}^2$.  The five DB profiles separate only in the immediate neighborhood of the jump. Fig.~\ref{fig:exp2_profiles}(b) reports same-scheme self-convergence at $T=0.005$.
We take $N=2^k$, $k=4,\ldots,8$, and keep $\Delta t=10^{-6}$ fixed across the family.  Each
coarse solution is compared with the same scheme on the doubled grid; this avoids limiting
IWPE by the error of a first-order reference scheme. Because the cell-centered grids do not nest
under refinement, fine-grid values are transferred to the coarse grid with a local four-point Lagrange interpolant, using a one-sided stencil at the jump.  Its $O(h_{\rm ref}^4)$ interpolation error remains below the discretization errors reported below.
The early final time also prevents the
common discrete equilibrium of the DB schemes from masking their transient differences.
On the finest pair, the $\ell^1$ orders are $2.00$ for IWPE, $1.06$--$1.10$ for SG/WPE,
ED, AM-$e^{-V}$, and LM-$e^{-V}$, and $0.99$ for upwind. 

\begin{figure}[htbp]
 \centering
 \begin{overpic}[width= 0.56 \textwidth]{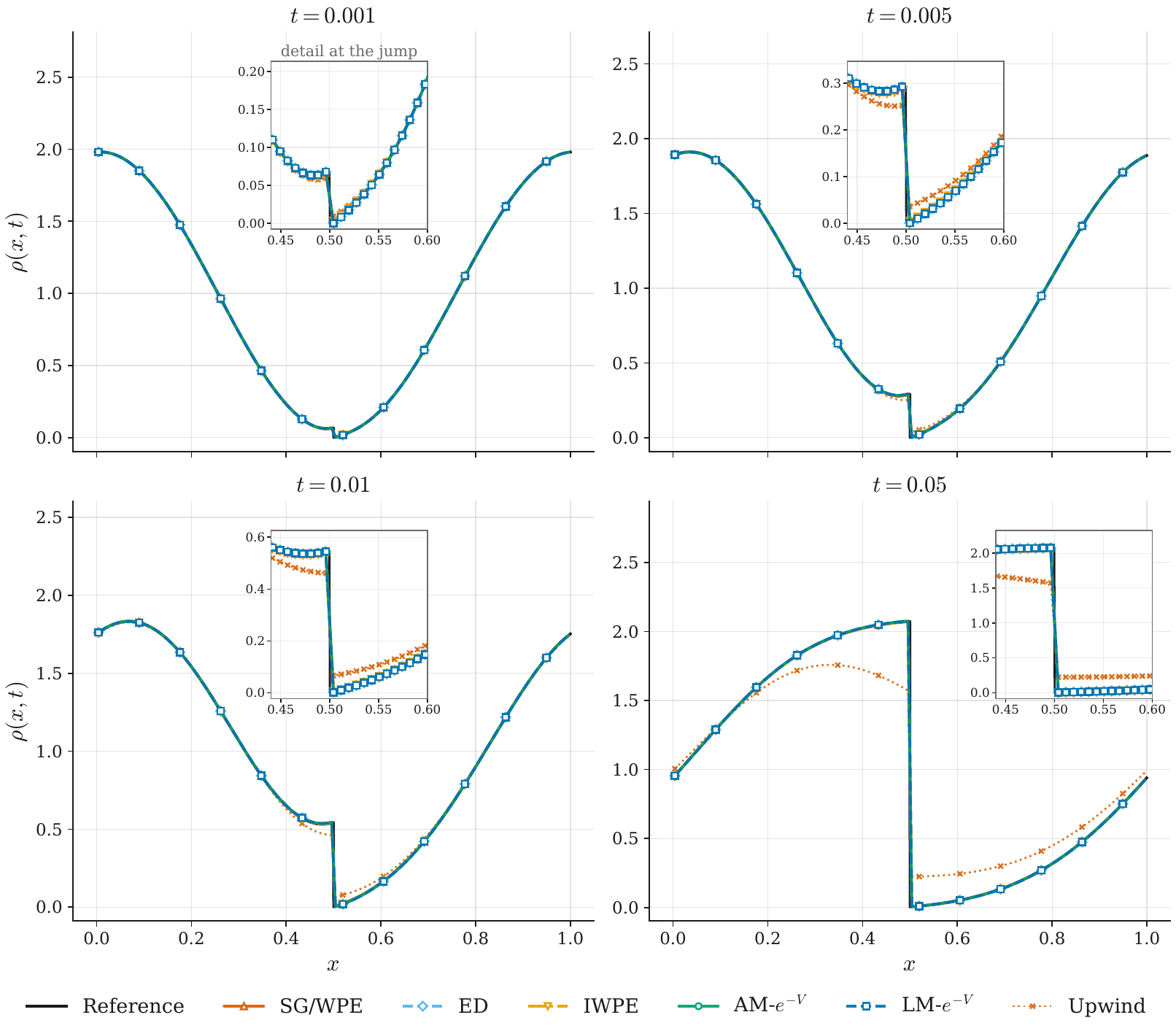}
    \put(-5, 82){(a)}
 \end{overpic}
 \hfill
 \begin{overpic}[width= 0.4 \textwidth]{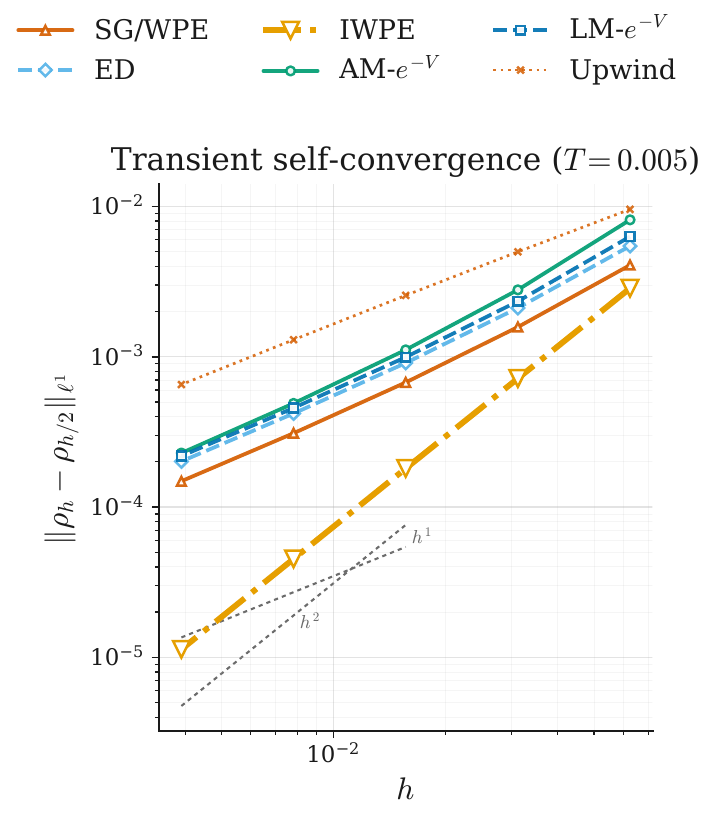}
    \put(-3, 100){(b)}
 \end{overpic}
 \caption{(a) solution profiles at $t=0.001,0.005,0.01,0.05$ for $N=128$
 against the IWPE reference at $N_{\rm ref}=2048$, with insets magnifying the jump at
 $x=1/2$; $\Delta t=h^2$. (b) Transient spatial accuracy at $T=0.005$ by same-scheme self-convergence,
 $N\in\{16,\ldots,256\}$ with $\Delta t=10^{-6}$ fixed across the family.}
 \label{fig:exp2_profiles}
\end{figure}

To test the long-time behavior, we use $N=256$, $T=0.75$, and $\Delta t=10^{-5}$.  The continuum
reference is $\pi=e^{-V}/Z$, with $Z=6.689334$ computed by splitting the quadrature at the
jump.  Fig.~\ref{eq_disc}(a)--(c) shows that the DB schemes accurately reproduce the
discontinuous Gibbs profile and dissipate the discrete free energy monotonically.  Their
$\ell^1$ error against $\pi$ is $4.46\times10^{-8}$.  Upwind also dissipates energy but settles above
$\mathcal F_h[\pi_h]$, with an $\ell^1$ equilibrium error of $3.48\times10^{-1}$.  Although
only IWPE is second-order accurate during the transient, all five DB schemes have the same equilibrium
accuracy because they share the same discrete Gibbs state $\pi_h$.

Fig.~\ref{eq_disc}(d) refines this equilibrium error in $h$ over $N\in\{32,64,128,256\}$
at the same $T=0.75$ and $\Delta t=10^{-5}$.  In contrast with Experiment~1, the $O(h^2)$ rate
now appears: all five DB schemes yield almost the same errors corresponding to an observed order of $2.00$, because the Euler--Maclaurin argument behind the spectral accuracy of $Z_h$
fails once $V$ has a jump.  The upwind error stalls at $3.58\times10^{-1}$--$3.48\times10^{-1}$
(order $0.01$) and its $\ell^\infty$ error even grows slightly under refinement, from
$1.046$ to $1.068$, because the error in the upwind ratio $F_{i+\frac12}/B_{i+\frac12}$ relative to its detailed-balance value is $O(1)$ and does not vanish under mesh refinement at a discontinuity of $V$, as discussed in Section~\ref{subsec:trad_disc}. 

\begin{figure}[!h]
\includegraphics[width = \textwidth]{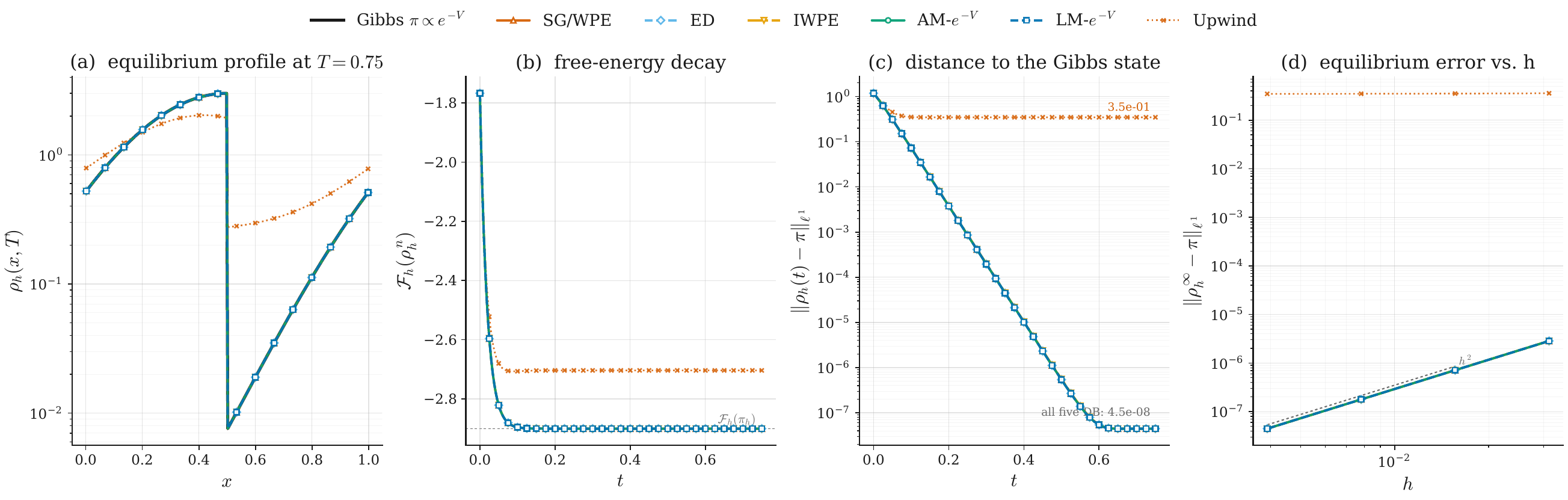}
\caption{Long-time behavior for the discontinuous potential at $N=256$,
$T=0.75$, $\Delta t=10^{-5}$.  (a) Solutions at $T=0.75$ against the continuum Gibbs state
(logarithmic scale).  (b) Discrete free energy, with $\mathcal F_h[\pi_h]$ marked.
(c) $\ell^1$ distance to the Gibbs state as a function of time.  (d) The same equilibrium
error refined in $h$ over $N\in\{32,\ldots,256\}$.}\label{eq_disc}
\end{figure}

%-----------------------------------------------------------------------
\subsection{Nonlinear mobility and saturation}\label{subsec:exp3}
%-----------------------------------------------------------------------

For the nonlinear-mobility extension, we consider the
one-dimensional saturation model \cite{carrillo2024structure} 
\begin{equation}\label{eq:saturation_model}
    \partial_t\rho
    =
    \partial_x\left[
        \rho(\alpha-\rho)
        \partial_x\left(
            \ln\rho+\frac{1}{2}x^2
        \right)
    \right],
    \qquad x\in[-4,4],
\end{equation}
subject to no-flux boundary conditions. The corresponding free energy
and mobility are
\begin{equation}\label{eq:saturation_energy}
    \mathcal F[\rho]
    =
    \int_{-4}^{4}
    \left[
        \rho(\ln\rho-1)
        +\frac{1}{2}x^2\rho
    \right]\dd x,
    \qquad
    m(\rho)=\rho(\alpha-\rho).
\end{equation}
The physically admissible state space is
$0\leq\rho\leq\alpha$. We take $\alpha=1$
and use the uniform initial condition $\rho_0(x)=\frac{M}{8}$.

The equilibrium is the minimizer of \eqref{eq:saturation_energy}
under the mass constraint $\int_{-4}^{4}\rho\,\dd x=M$ and the
saturation constraint $\rho\leq\alpha$ \cite{carrillo2024structure}. Define
\[
    Z_\Omega
    =
    \int_{-4}^{4}
    \exp\left(-\frac{1}{2}x^2\right)\dd x,
    \qquad
    M_c=\alpha Z_\Omega.
\]
For $M\leq M_c$, the constraint is inactive and the equilibrium is
the Gibbs state
\begin{equation}\label{eq:saturation_subcritical}
    \rho_\infty(x)
    =
    \frac{M}{Z_\Omega}
    \exp\left(-\frac{1}{2}x^2\right).
\end{equation}
For $M>M_c$, the constraint becomes active and the equilibrium
contains a saturated plateau:
\begin{equation}\label{eq:saturation_supercritical}
    \rho_\infty(x)
    =
    \alpha
    \exp\left[
        -\frac{1}{2}(x^2-\ell^2)_+
    \right],
\end{equation}
where $\ell>0$ is determined by $\int_{-4}^{4}\rho_\infty(x)\,\dd x=M$.
Both branches are normalized over $[-4,4]$ rather than over $\mathbb R$, which is the
relevant convention here because the numerical solutions conserve mass on the domain; this gives
$M_c=2.50647$ on $[-4,4]$, against the whole-line value $\alpha\sqrt{2\pi}=2.50663$.
We consider both a subcritical case, $M=2$, and a supercritical case,
$M=3.32$, as in \cite{carrillo2024structure}. For the latter, the plateau half-width is $\ell=1.00678$, and $\rho_\infty=\alpha$ on $[-\ell,\ell]$.

To apply the construction of Section~\ref{subsec:nonlinear_mob}, we write $a(\rho)=m(\rho)/\rho=\alpha-\rho$. To preserve nonnegativity, we set the nodal kinetic coefficients to
    $a_i^n=(\alpha-\rho_i^n)_+$
and reconstruct them at the interfaces using the harmonic mean,
\begin{equation}\label{eq:saturation_harmonic}
    \widehat a_{i+\frac12}^n
    =
    \begin{cases}
    \displaystyle
    \frac{2a_i^n a_{i+1}^n}{a_i^n+a_{i+1}^n},
        & a_i^n+a_{i+1}^n>0,\\[2mm]
    0,  & a_i^n+a_{i+1}^n=0.
    \end{cases}
\end{equation}
The resulting edge coefficient is then combined with the potential
reconstruction of each DB scheme as in \eqref{rates_nonlinear_M_product}. All
mobilities are frozen at time level $n$ and the density is advanced
using the semi-implicit update \eqref{SI_nonlinear_M}.  The positive-part operator defining $a_i^n$
prevents any numerical overshoot from producing a negative kinetic coefficient and turning
the generator into a backward diffusion.

\begin{figure}[htbp]
 \centering
 \includegraphics[width=\textwidth]{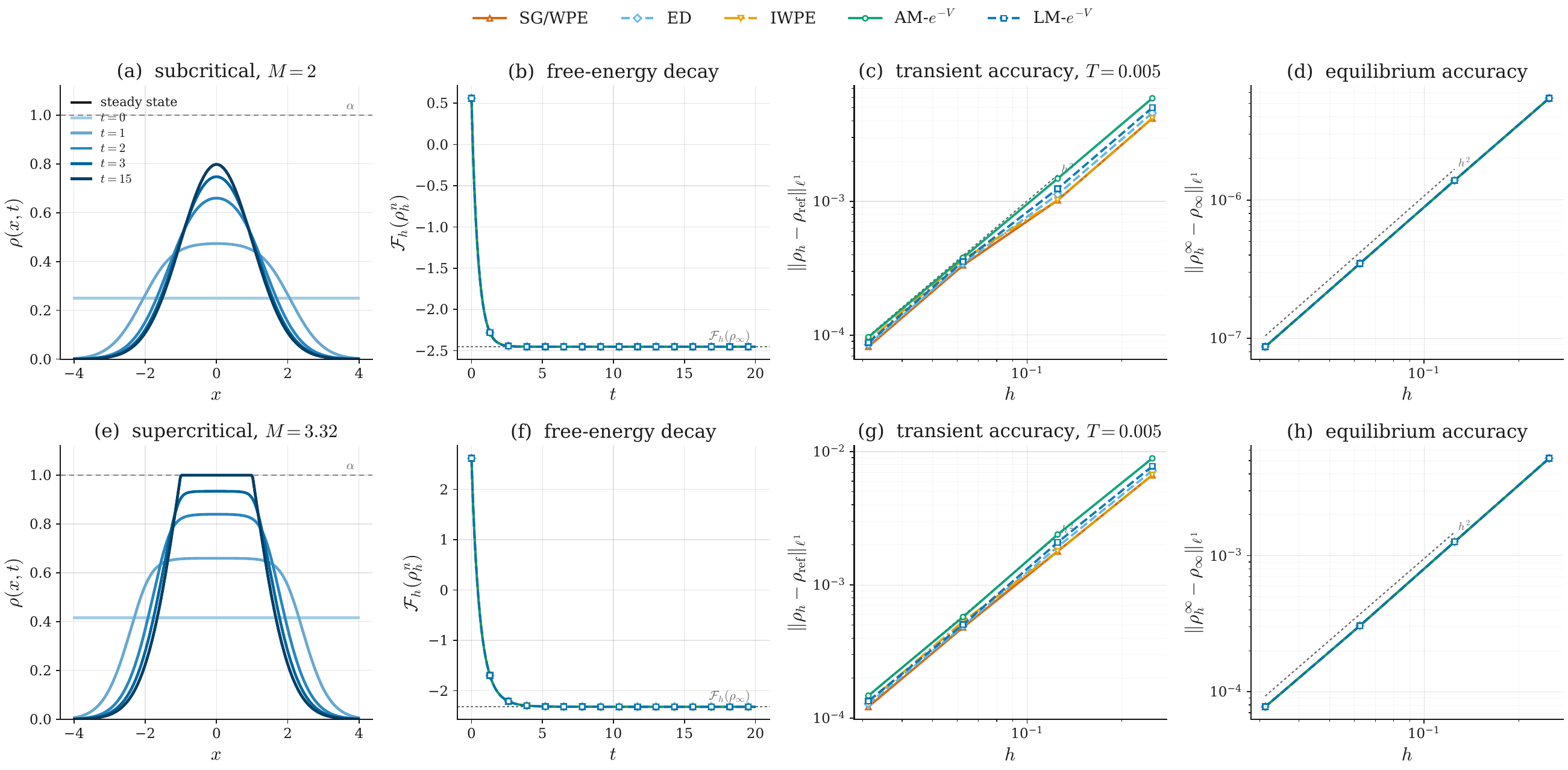}
 \caption{The saturation model \eqref{eq:saturation_model} in the subcritical
 regime $M=2$ (top row) and the supercritical regime $M=3.32$ (bottom row).
 (a), (e): AM-$e^{-V}$ snapshots at $t=0,1,2,3,15$ on $N=256$ with $\Delta t=0.01$, against the
 analytic steady state and the bound $\alpha$.  (b), (f): Discrete free energy for the five DB
 schemes, with $\mathcal F(\rho_\infty)$ marked. (c), (g): Transient accuracy at $T=0.005$,
 $N\in\{32,\ldots,256\}$, $\Delta t=10^{-5}$ fixed, against a spline-sampled reference at
 $N_{\rm ref}=1024$.  (d), (h): Equilibrium accuracy at $T=20$ with $\Delta t=0.01$ against
 \eqref{eq:saturation_subcritical}--\eqref{eq:saturation_supercritical}.}
 \label{fig:saturation}
\end{figure}

In this numerical test, the grid is cell-centered, $x_i=-4+(i+\tfrac12)h$ with $h=8/N$, so that the $N$ cells carry
exactly $N-1$ interior edges; the no-flux condition is imposed by setting the forward and
backward rates of the wrap-around edge to zero.
Fig.~\ref{fig:saturation} reports the results.  The dynamics in
Fig.~\ref{fig:saturation}(a) and (e) use $N=256$, $\Delta t=0.01$, and $T=20$, with snapshots at
$t=0,1,2,3,15$.  For the transient test,
Fig.~\ref{fig:saturation}(c) and (g), the error at $T=0.005$ is measured on
$N\in\{32,64,128,256\}$ with $\Delta t=10^{-5}$ held fixed across the family, against a
fine-grid reference computed with the SG/WPE scheme at $N_{\rm ref}=1024$ and
$\Delta t=5\times10^{-7}$; because cell-centered grids do not nest under refinement by two, that
reference is sampled by a cubic spline, whose $O(h_{\rm ref}^4)$ error is far below the
coarse-grid errors.  The final time is chosen early enough that $\rho$ is still below $\alpha$
in both mass regimes, so that this test measures the smooth transient rather than the
constrained one.  The equilibrium test, Fig.~\ref{fig:saturation}(d) and (h), takes the
solution at $T=20$ with $\Delta t=0.01$ on the same grids and compares it with
\eqref{eq:saturation_subcritical} or \eqref{eq:saturation_supercritical}.

In the subcritical case, the solution converges to the Gaussian Gibbs state and its peak
remains $2.0\times10^{-1}$ below $\alpha$.  In the supercritical case, the schemes capture
the saturated plateau and the Gibbs-type tails, with
$\max_i\rho_i\leq\alpha$ up to round-off for every DB scheme.  Across both cases, the minimum
density is at least $2.8\times10^{-4}$, the mass error is below $2.5\times10^{-14}$, and the
largest positive free-energy increment is $8.9\times10^{-16}$.  On the finest pair, the
transient $\ell^1$ orders are $1.91$--$2.02$ and the equilibrium orders are
$1.98$--$2.00$ across the five DB schemes and both mass regimes.

The harmonic reconstruction is essential near the plateau: if either endpoint approaches
$\alpha$, the edge mobility approaches zero and suppresses further transport into the
saturated node.  An arithmetic reconstruction remains positive when only one endpoint is
saturated and can therefore generate grid-scale oscillations and overshoots.  Thus the edge
reconstruction affects both transient accuracy and preservation of the admissible state. It is worth mentioning that the harmonic reconstruction alone does not make the frozen-mobility update unconditionally
upper-bound preserving: an excessively large step can move $\rho_i^{n+1}$ above $\alpha$
before the adjacent edge mobilities are deactivated. In all runs reported here we take $\Delta t = 0.01$ with $h = 8/N, N \in \{32, \ldots, 256\}$, and no violation of the saturation bound is observed.

%-----------------------------------------------------------------------
\subsection{Nonlocal aggregation--diffusion}\label{subsec:exp4}
%-----------------------------------------------------------------------
In this subsection, we consider an example of a nonlocal aggregation--diffusion equation with a smooth interaction kernel. We set the external potential to zero. Since we consider periodic boundary conditions on the interval \([-1,1]\), we use the attractive interaction kernel obtained by periodizing a Gaussian:
\begin{equation}\label{eq:exp3_kernel}
K_{\rm per}(r)
=
-10\sum_{m\in\mathbb Z}
\exp \left(-\frac{(r+2m)^2}{2\sigma^2}\right),
\qquad
\sigma=0.2.
\end{equation}
In the numerical computation, we truncate the sum to \(m=-4,\ldots,4\).
For this value of \(\sigma\), all neglected terms are smaller than machine
precision. We use the periodized Gaussian rather than evaluating the
Gaussian at the shortest periodic distance, because the periodized kernel
retains the nonpositive Fourier coefficients of the original attractive
Gaussian. 
The initial condition is the unit-mass sum of two Gaussian bumps centered at
$x=\pm0.4$, each with standard deviation $0.07$.

At every step, we compute $\widetilde V_i^n=(\mathcal K_h\bm\rho^n)_i$ by FFT, assemble the
generator $A_h^n$ from \eqref{eq:lagged_rates}, and solve
$(I-\Delta t\,A_h^n)\bm\rho^{n+1}=\bm\rho^n$.  The main run uses
$N=256$ ($h=1/128$), $\Delta t=10^{-5}$, and $T=4$.  Since no closed-form equilibrium is
available, the reference state is the fixed point of the discretization itself,
$\bm\rho=e^{-\mathcal K_h\bm\rho}/Z_h$, computed by Picard iteration with damping $0.5$ from the same
initial condition used in the time-dependent runs; it converges in $85$ iterations to a round-off-level residual.  Starting the iteration from $\rho_0$ matters, because the energy is not
convex on a periodic domain and every translate of a bump, as well as the uniform state, is
also an equilibrium.
As an independent check of our computations, we also run the first-order nonlinear upwind scheme \cite{carrillo2015finite}, which upwinds the chemical potential
$\ln\bm\rho+\mathcal K_h\bm\rho$.  Its chemical potential is lagged and the resulting operator is advanced with backward Euler. 

\begin{figure}[htbp]
 \centering
 \includegraphics[width=\textwidth]{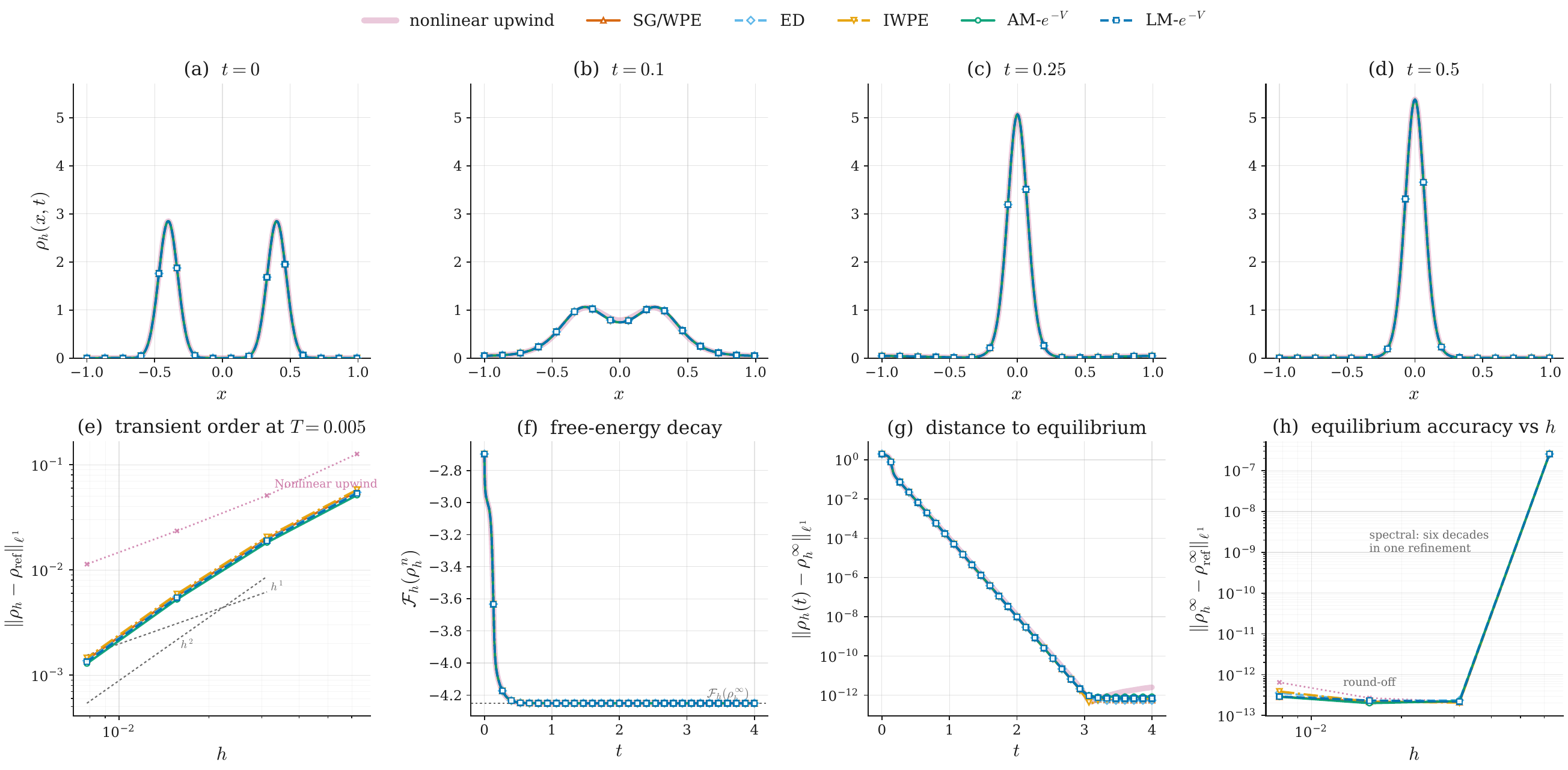}
 \caption{Nonlocal aggregation--diffusion at $N=256$, $\Delta t=10^{-5}$.
 (a)--(d) The merger at $t=0,0.1,0.25,0.5$ for the five DB schemes, with the nonlinear
 upwind scheme of \cite{carrillo2015finite} drawn as a pale underlay.  (e) Transient spatial order
 at $T=0.005$ against an ED reference at $N_{\rm ref}=1024$; the DB schemes use $\Delta t=h^2$, and upwind uses $\Delta t=0.4h^2$.  (f) Free-energy decay to $\mathcal F_h[\rho_h^\infty]$.  (g) Distance
 to the discrete fixed point.  (h) The same equilibrium error refined in $h$: spectrally
 accurate, hence no order in $h$.}
 \label{fig:exp3_merging}
\end{figure}

Fig.~\ref{fig:exp3_merging}(a)--(d) shows the merger at $t=0,0.1,0.25,0.5$; the five DB
schemes and nonlinear upwind are visually indistinguishable.  All runs preserve positivity
and conserve mass.  At $T=4$, the $\ell^1$ distances to the Picard fixed point range from
$5.52\times10^{-13}$ to $2.53\times10^{-12}$ across all six methods, as shown in Fig.~\ref{fig:exp3_merging}(g). Fig.~\ref{fig:exp3_merging}(e) measures transient spatial accuracy at $T=0.005$ on
$N\in\{32,64,128,256\}$ against an ED reference at $N_{\rm ref}=1024$, restricted by exact
sampling at coincident nodes.  The DB schemes use $\Delta t=h^2$, while nonlinear upwind
uses $\Delta t=0.4h^2$; the smaller step is required for the upwind iteration to reach the
correct branch.  On the finest pair, the DB orders are $2.02$--$2.03$ and the upwind order
is $1.05$. Fig.~\ref{fig:exp3_merging}(h) refines the \emph{equilibrium error} over the same grids at $T=4$ with
$\Delta t=h^2$, against the fixed point at $N_{\rm ref}=1024$.  Neither family exhibits an algebraic convergence order, and this is again the expected behavior: both converge to $\bm\rho=e^{-\mathcal K_h\bm\rho}/Z_h$, a
Nystr\"om discretization of the continuum fixed-point equation by the periodic trapezoidal
rule, which is spectrally accurate for this smooth kernel.  The error drops six decades
between $N=32$ and $N=64$ and then sits at round-off. % the first-order transient rate of the upwind does not determine the equilibrium accuracy, because it characterizes the transient trajectory rather than the limiting state.

%-----------------------------------------------------------------------
\subsection{Two-dimensional case}\label{subsec:exp5}
%-----------------------------------------------------------------------

Finally, we consider a two-dimensional example. On the periodic square $[0,1]^2$, we use a banana-shaped potential,
\begin{equation}\label{eq:exp4_potential}
\begin{aligned}
 V(x,y)&=0.8\,[1-\cos(2\pi(x-\tfrac12))]
 +10\,[1-\cos(2\pi(y-y_c(x)))],
\end{aligned}
\end{equation}
where $y_c(x) =\tfrac14+0.18\,[1-\cos(2\pi(x-\tfrac12))]$ and initialize
$\rho_0(x,y)=Z_0^{-1}\exp[-60((x-1/2)^2+(y-1/2)^2)]$, a Gaussian centered away from the
minimum of $V$.
On an $N\times N$ grid, the generator is assembled as the sum of horizontal
and vertical edge generators,
\begin{equation}
 A_h=A_h^x+A_h^y,
 \qquad
 F^x_{i+1/2,j}=h^{-2} \psi(V_{i+1,j}-V_{i,j}),
 \quad
 F^y_{i,j+1/2}=h^{-2} \psi(V_{i,j+1}-V_{i,j}),
\end{equation}
with the backward rates obtained by reversing the potential differences. Here $\psi$ is the
rate function of the chosen reconstruction (see (\ref{eq:psi_rates})).

\begin{figure}[htbp]
 \centering
 \includegraphics[width=\textwidth]{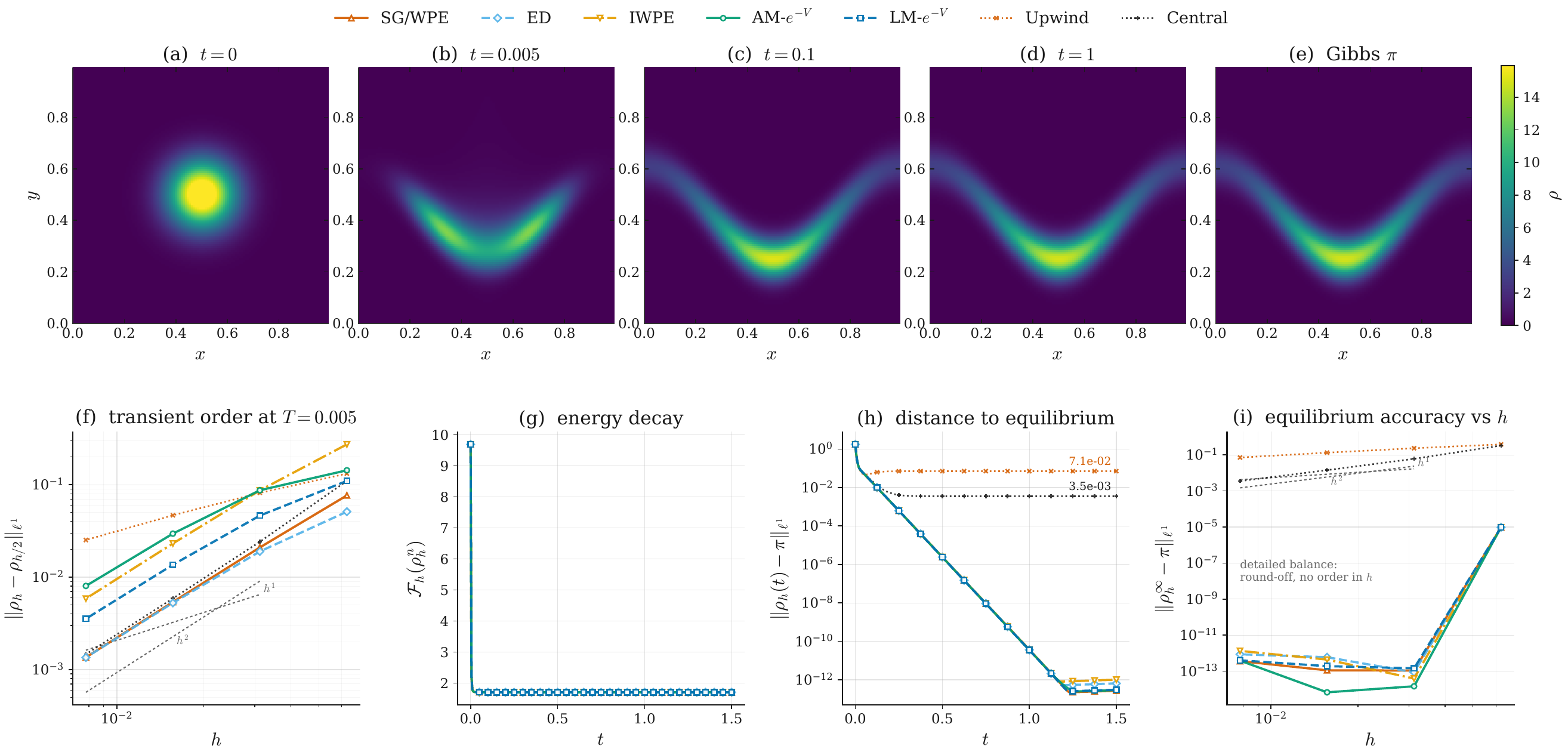}
 \caption{Two-dimensional example.  (a)--(d) AM-$e^{-V}$ snapshots at $t=0,0.005,0.1,1$ on
 $N=128$ with $\Delta t=10^{-3}$, and (e) the discrete
 Gibbs state.  (f) Transient order by same-scheme self-convergence at $T=0.005$ with
 $\Delta t=10^{-5}$ fixed.  (g) Free-energy decay.  (h) Distance to the Gibbs state.
 (i) Equilibrium error refined in $h$ at $T=2$ against the continuum Gibbs state.}
 \label{fig:exp4_snapshots}
\end{figure}

The relaxation run uses $N=128$, $\Delta t=10^{-3}$, and $T=1.5$; the snapshots
in Fig.~\ref{fig:exp4_snapshots}(a)--(d) are taken at $t=0,0.005,0.1,1$ using AM-$e^{-V}$ and
are compared with the discrete Gibbs state in Fig.~\ref{fig:exp4_snapshots}~(e). Fig. ~\ref{fig:exp4_snapshots}~(h) shows the distance to the Gibbs state for all schemes. As in the one-dimensional case with a smooth potential, the DB schemes approximate the Gibbs state more accurately than the upwind and central difference schemes. At $T=1.5$, the DB errors are at
round-off, while the upwind and central difference errors are $7.1\times10^{-2}$ and
$3.5\times10^{-3}$, respectively.   Fig.~\ref{fig:exp4_snapshots}(f) measures same-scheme self-convergence by comparing the solution on an $N\times N$ grid at coincident nodes with the solution obtained using the same method on a $2N\times2N$ grid, for
$N\in\{16,32,64,128\}$.  We use $T=0.005$ and keep $\Delta t=10^{-5}$ fixed across the
family. All DB schemes show second-order accuracy, while upwind is first-order accurate and central differences are second-order accurate.
Fig.~\ref{fig:exp4_snapshots}(i) refines the equilibrium error in $h$ over $N\in\{16,32,64,128\}$ at $T=2$ with
$\Delta t=10^{-3}$, against the continuum Gibbs state, whose normalization is computed
by the midpoint rule on a $1024^2$ grid, since $V$ is smooth and
periodic.  As in Experiment~1, the DB schemes sit at round-off with no order in $h$, while
upwind has an observed order of $0.90$ and central differences have an observed order of $2.02$.

\section{Conclusion}\label{sec:conclusion}
We developed a variational--Markov construction of detailed-balance
master-equation discretizations for Fokker--Planck equations.  Its main
feature is the separation of three roles on every edge in the resulting master equation: the projected
energy--dissipation law fixes the net flux, detailed balance fixes the ratio
of the two directional rates, and a local reconstruction fixes their common
symmetric coefficient.  This separation makes precise why transient and
equilibrium accuracy are distinct.  Detailed balance determines the
discrete Gibbs state and guarantees free-energy decay, whereas consistency
of the symmetric coefficient with the continuum mobility controls the
transient dynamics. The exact density-dependent rates reproduce the
projected dissipation edge by edge. In the linear fixed-potential setting, the practical density-independent rates
retain exact detailed balance and approximate the continuum
energy--dissipation law to second order for smooth positive solutions.

The ED, SG/WPE, IWPE, AM-$e^{-V}$, and LM-$e^{-V}$ rates all arise from
second-order reconstructions within this construction.  Each satisfies the
multiplicative symmetry associated with detailed balance and therefore has
the sampled Gibbs state as an exact stationary state.  Among these five
choices, SG/WPE is distinguished by also satisfying the additive
consistency relation of the classical $B$-scheme formulation.  Thus the
framework both recovers established exponential-fitting schemes and
identifies the structural choice made by each reconstruction.

We further developed fully discrete schemes.
For a fixed potential, backward Euler gives a linear, uniquely solvable
update that preserves mass and nonnegativity and dissipates the discrete
free energy for every time-step size. For nonlinear mobilities, freezing a state-dependent mobility
preserves these properties, including when some edge mobilities vanish.  For nonlocal interactions, the lagged-potential semi-implicit treatment also requires only
a linear solve per step and is unconditionally mass-conservative and
positivity preserving.  It dissipates the true nonlocal energy
unconditionally for a negative-semidefinite interaction kernel; for a
general symmetric kernel, the exact energy-increment decomposition isolates
the quadratic remainder and yields a kernel-dependent time-step condition.
The edgewise construction further extends to general $f$-divergence
energies and Cartesian grids in multiple dimensions. The numerical experiments confirm both the structural guarantees and the
distinction between transient and equilibrium accuracy in the smooth and discontinuous potential, nonlinear mobility, and nonlocal interaction cases, and in two dimensions. 

Several directions remain open. A natural next step is to develop a
convergence theory for the variational--Markov discretizations, particularly
for discontinuous potentials, nonlinear mobilities, and nonlocal
interactions. Because the discrete free energy and dissipation are
constructed together, the analysis may be carried out at the
energy--dissipation level: uniform discrete energy and dissipation estimates
may provide the stability and compactness needed to pass to the continuum
limit. Another direction is to adapt the same construction to sampling
methods for high-dimensional probability distributions, where
tensor-product grids are infeasible. In this setting, the detailed-balance
ratio can be used to preserve the target distribution, while the symmetric
component of the transition rates can be designed to control, and potentially
accelerate, the sampling dynamics.

\section*{Acknowledgments}
The authors would like to thank Professor Chun Liu for helpful discussions. This work was partially supported by NSF DMS-2410740.

\section*{Declaration of generative AI}

During the preparation of this work, the authors used ChatGPT (OpenAI) and Claude (Anthropic) to improve the clarity and language of the manuscript and to assist with the development of numerical code. The authors reviewed, edited, tested, and verified all AI-assisted outputs and take full responsibility for the content and results presented in the article.

\bibliographystyle{unsrt} 
\bibliography{FP}

\end{document}